\documentclass[11pt,letterpaper,reqno]{amsart}

\usepackage[T1]{fontenc}
\usepackage[utf8]{inputenc}
\usepackage{newtxtext,newtxmath}
\usepackage{microtype}
\usepackage{mathtools}
\usepackage{enumitem}
\usepackage{xcolor}
\usepackage{url}
\usepackage[colorlinks=true,linkcolor=blue!55!black,citecolor=green!45!black,urlcolor=blue!55!black]{hyperref}
\hypersetup{
  pdftitle={Torsion-Shifted Ribet Points and Cofinite Prime-Index Returns},
  pdfauthor={Khai-Hoan Nguyen-Dang},
  pdfsubject={Arithmetic returns, Ribet points, and geometric divisibility sequences over number fields},
  pdfkeywords={Silverman conjecture, Ribet point, semiabelian variety, reduction order, geometric divisibility sequence, N\'eron model}
}
\usepackage[nameinlink,noabbrev,capitalise]{cleveref}

\setlist[enumerate]{leftmargin=2.25em,itemsep=2pt,topsep=4pt}
\setlist[itemize]{leftmargin=2.1em,itemsep=2pt,topsep=4pt}
\allowdisplaybreaks
\newtheorem{theorem}{Theorem}[section]
\newtheorem{proposition}[theorem]{Proposition}
\newtheorem{corollary}[theorem]{Corollary}
\newtheorem{lemma}[theorem]{Lemma}
\newtheorem{conjecture}[theorem]{Conjecture}
\theoremstyle{definition}
\newtheorem{definition}[theorem]{Definition}

\theoremstyle{remark}
\newtheorem{remark}[theorem]{Remark}

\crefname{conjecture}{conjecture}{conjectures}
\Crefname{conjecture}{Conjecture}{Conjectures}
\crefname{question}{question}{questions}
\Crefname{question}{Question}{Questions}

\newcommand{\Gm}{\mathbf G_m}

\newcommand{\Q}{\mathbf Q}
\newcommand{\Z}{\mathbf Z}
\newcommand{\N}{\mathbf N}
\newcommand{\F}{\mathbf F}
\newcommand{\OK}{\mathcal O_K}

\newcommand{\ord}{\operatorname{ord}}
\newcommand{\End}{\operatorname{End}}
\newcommand{\Hom}{\operatorname{Hom}}
\newcommand{\Ext}{\operatorname{Ext}}

\newcommand{\Spec}{\operatorname{Spec}}
\newcommand{\Spf}{\operatorname{Spf}}
\newcommand{\Gal}{\operatorname{Gal}}
\newcommand{\im}{\operatorname{im}}
\newcommand{\id}{\operatorname{id}}

\newcommand{\supp}{\operatorname{Supp}}

\newcommand{\rad}{\operatorname{rad}}
\newcommand{\dd}{\mathfrak d}
\newcommand{\cO}{\mathcal O}

\newcommand{\cP}{\mathcal P}
\newcommand{\tors}{\mathrm{tors}}
\newcommand{\Zar}{\mathrm{Zar}}
\newcommand{\eps}{\varepsilon}
\newcommand{\Ner}{\mathcal N}
\newcommand{\SilvermanConj}{\hyperref[conj:silverman]{Conjecture~\ref*{conj:silverman}}}

\DeclareMathOperator{\length}{length}
\DeclareMathOperator{\Lie}{Lie}

\title[Ribet Points, geometric divisibility sequence and order of reductions]{Ribet Points, geometric divisibility sequence and order of reductions on semiabelian varieties}

\author{Khai-Hoan Nguyen-Dang}
\address{Morningside Center of Mathematics, Chinese Academy of Sciences, No. 55 Zhongguancun East Road, Beijing 100190, China}
\email{khaihoann@gmail.com}


\date{}

\begin{document}

\begin{abstract}
Silverman conjectured that the geometric divisibility sequence attached to a Zariski-dense point on an irreducible commutative algebraic group of dimension at least two, with no unipotent part, returns to its initial value infinitely often. We construct, for the first time, unconditional examples of this phenomenon on geometrically nonsplit semiabelian varieties over number fields using torsion translates of normalized Ribet points.

More precisely, let $A/K$ be a positive-dimensional abelian variety over number field $K$, let
\[
1\longrightarrow\Gm\xrightarrow{\iota}G_q\xrightarrow{\pi}A
 \longrightarrow0
\]
be the extension represented by $q\in A^\vee(K)$, and let
$R_\beta(q)\in G_q(K)$ be the normalized Ribet point associated with a
homomorphism $\beta:A^\vee\to A$.  We set
$\delta:=\beta-\widehat\beta$ and assume that $\delta$ is an isogeny and that
$\mathbf Z(\delta q)$ is Zariski dense in $A$.  For a torsion point
$t\in\Gm(K)_{\tors}$, identify $t$ with $\iota(t)$ and set $P=R_\beta(q)+t$. Then $P$ has Zariski-dense cyclic orbit in the geometrically nonsplit extension $G_q$.  There exists an explicit integer $N_{\delta,t}$ such that if
$N_{\delta,t}>1$, then $N_{\delta,t}\mid \ord(\overline P_v)$ at all but finitely many places $v$.  Consequently, there is a squarefree integer $Q_P>1$ such that
\[
 (n,Q_P)=1
 \quad\Longrightarrow\quad
 \mathfrak d_{\mathcal N}(nP)=\mathfrak d_{\mathcal N}(P),
\]
where $\mathfrak d_{\mathcal N}$ denotes the full denominator ideal on the
Néron lft-model $\mathcal N$. In particular, we construct explicitly a geometrically nonsplit semiabelian surface $G/\mathbf Q$ and a semiabelian threefold over $\mathbf Q$ satisfying the Silverman conjecture. It follows that we can construct instances satisfying the Silverman conjecture for every dimension at least two. 

Over suitable number fields, every squarefree integer occurs as the exact eventual reduction divisor of a Zariski-dense point on a geometrically nonsplit semiabelian threefold. Moreover, the exact-order and primitive-divisor spectra can be made arbitrarily sparse.
\end{abstract}

\maketitle
\enlargethispage{3pt}

\section{Introduction}\label{sec:introduction}

\subsection{Geometric divisibility sequences and the
Silverman--Cheon--Hahn--Stange theorem}

A sequence $\mathbf a=(a_n)_{n\ge1}$ of nonzero integers is a
\emph{divisibility sequence} if
\[
 m\mid n\quad\Longrightarrow\quad a_m\mid a_n.
\]

The classical examples $a^n-b^n$, Lucas and Lehmer sequences, and Ward's elliptic divisibility sequences already bring together recurrence sequence theory and Diophantine geometry
\cite{Ward1948,EverestPoortenShparlinskiWard2003}.  For a prime $p$ dividing at least one term, its \emph{rank of apparition} is
\[
 r_p(\mathbf a):=\min\{n\ge1:p\mid a_n\}.
\]
The prime $p$ is a \emph{primitive divisor} of $a_n$ if
$r_p(\mathbf a)=n$, and the \emph{Zsigmondy set} is
\[
 \mathcal Z(\mathbf a)
 :=\{n\ge1:a_n\text{ has no primitive prime divisor}\}.
\]

Zsigmondy's theorem and Schinzel's number-field extension show that $a^n-b^n$ has a
primitive divisor at every sufficiently large index, subject to the classical
finite list of exceptions \cite{Zsigmondy1892,Schinzel1974}.  For Lucas and
Lehmer sequences, Bilu--Hanrot--Voutier proved the uniform bound $n>30$
\cite{BiluHanrotVoutier2001}.  More recent work studies structural
classification and sieve theory for divisibility sequences
\cite{Granville2022,BrowningVerzobio2024}.

Divisibility sequences encode multiplication in algebraic groups, and
Silverman's construction makes this geometry literal.  Let $\mathcal G$ be a
separated commutative group scheme over an arithmetic base, let $e$ be its
identity section, let $\mathcal I_e$ be the ideal defining $e$, and let $P$
have infinite order.  The pullbacks
\[
 \mathfrak d_{\mathcal G}(nP):=([n]P)^*\mathcal I_e
\]
form an ideal-valued divisibility sequence, called the denominator sequence
\cite[Definition~4 and Proposition~8]{Silverman2005}.  At a finite place $v$
of good finite-type reduction, put
\[
 d_v(P):=\ord(\overline P_v).
\]
Then
\[
 v\!\left(\mathfrak d_{\mathcal G}(nP)\right)>0
 \quad\Longleftrightarrow\quad
 [n]\overline P_v=e_v
 \quad\Longleftrightarrow\quad
 d_v(P)\mid n,
\]
and hence
\[
 d_v(P)=
 \min\left\{
 n\ge1:
 v\!\left(\mathfrak d_{\mathcal G}(nP)\right)>0
 \right\}.
\]
Thus $d_v(P)$ is the first return time to the identity modulo $v$.  The place
$v$ is primitive at index $n$ exactly when $d_v(P)=n$, while
$v(\mathfrak d_{\mathcal G}(nP))$ records the scheme-theoretic depth of that
return.  Denominator ideals, ranks of apparition, primitive divisors, and
exact reduction orders are therefore complementary records of the same local
phenomenon.  This viewpoint underlies elliptic nets and pairing formulas,
Diophantine definability, and arithmetic dynamics
\cite{Stange2007,Stange2011,LauterStange2009,Poonen2002,
IngramSilverman2009,Stange2026}.

For elliptic curves over number fields, the exact-order problem is completed
by the following theorem.

\begin{theorem}[Silverman; Cheon--Hahn; Stange]
\label{thm:intro-silverman-cheon-hahn}
Let $E/K$ be an elliptic curve over a number field and let $Q\in E(K)$ be
non-torsion.  Then there is $n_0=n_0(E,Q)$ such that, for every $n\ge n_0$,
there is a finite place $v$ of good reduction with
\[
        d_v^E(Q):=\ord(\overline Q_v)=n.
\]
Equivalently, every sufficiently large term of the associated elliptic
denominator sequence has a primitive prime-ideal divisor.
\end{theorem}

Silverman proved the theorem over $\Q$ \cite[Proposition~10]{Silverman1988}, and Cheon--Hahn proved its
number-field extension
\cite[Main Theorem]{CheonHahn1999}.
The attribution to Stange records her corrected and complete formulation of the
local valuation input used in this circle of arguments (see~\cite[Theorem~6.1, Corollary~6.4, and Remark~6.5]{Stange2016Valuations}).

The theorem is the elliptic analogue of Bang--Zsigmondy and initiated work on
prime appearance, uniformity, and effectivity
\cite{EinsiedlerEverestWard2001,EverestShparlinski2005,
EverestMcLarenWard2006,Verzobio2023}.  It is natural to ask whether an
analogous exact-order theorem holds for higher-dimensional abelian or
semiabelian varieties.  The examples in this paper show that it fails
sharply in the semiabelian category, even for surfaces over $\Q$.  Primary
parts and prescribed valuations of reduction orders are studied in
\cite{Pink2004,Perucca2009}, while exact-order realization over global
function fields and in relative complex families is studied in
\cite{NguyenDangNguyen2026}.

\subsection{The Ailon--Rudnick problem and Silverman's conjecture}

A complementary problem asks not for a new divisor, but for an exact return to
the initial term.  For multiplicatively independent integers $a,b\ge2$, the
point $(a,b)\in\Gm^2(\Q)$ gives, away from primes dividing $ab$,
\[
 D_{n(a,b)}=\gcd(a^n-1,b^n-1).
\]
Bugeaud--Corvaja--Zannier~\cite{BugeaudCorvajaZannier2003} proved that, for every $\varepsilon>0$,
\[
 \gcd(a^n-1,b^n-1)\le \exp(\varepsilon n)
\]
for all sufficiently large $n$.  Corvaja--Zannier~\cite{CorvajaZannier2005} subsequently obtained quantitative refinements and broad
generalizations. Under the normalization
$\gcd(a-1,b-1)=1$, Ailon--Rudnick~\cite{AilonRudnick2004} conjectured that the gcd equals $1$
infinitely often and proved a stronger bounded-support theorem over
$\mathbf C[T]$.  Further structural reductions for
the integer sequence are developed in \cite{NguyenDangGCD2026}.

Silverman interpreted generalized gcds as heights on blowups.  Conditional on
Vojta's conjecture for the blowup of an abelian variety $A/\Q$ of dimension at
least two at the origin, he proved that a point with dense cyclic orbit
satisfies
\[
 \log D_{nP}\le \varepsilon n^2+O_{A,P,\varepsilon}(1)
\]
for every $\varepsilon>0$ \cite[Proposition~9]{Silverman2005}.
Related gcd--Vojta advances include \cite{Levin2019,Yasufuku2026}.
Motivated explicitly by the Ailon--Rudnick problem, Silverman formulated the
following exact recurrence conjecture \cite[Conjecture~10]{Silverman2005}.
Recall that for $P\in\mathcal G(\mathbf Z)$ one defines the positive integer
$D_P$ by
\[
        P^*\mathcal I_e=D_P\mathbf Z.
\]
Thus \(v_p(D_P)\) is the scheme-theoretic intersection multiplicity of the
sections \(P\) and \(e\) over \(p\).  In particular,
\[
        p\mid D_P
        \quad\Longleftrightarrow\quad
        \overline P_p=e_p.
\]
More generally,
\[
        D_{nP}\mathbf Z=([n]P)^*\mathcal I_e,
\]
so the equality \(D_{nP}=D_P\) means that every local contact multiplicity
of \(nP\) with the identity is exactly the same as that of \(P\).

\begin{conjecture}[Silverman]\label{conj:silverman}
Let $\mathcal G/\Z$ be a group scheme and let $P\in\mathcal G(\Z)$.  Assume
that the generic fiber $G=\mathcal G_{\Q}$ is an irreducible commutative
algebraic group of dimension at least two with no unipotent part and that
$\Z P$ is Zariski dense in $G$.  Then
\[
 D_{nP}=D_P
\]
for infinitely many positive integers $n$.
\end{conjecture}

The normalized Ailon--Rudnick conjecture is precisely the corresponding
split-torus special case.  Moreover, Silverman's return conjecture and a
naive higher-dimensional extension of the
Silverman--Cheon--Hahn theorem are mutually incompatible: a return
index cannot be an exact reduction order.

\begin{proposition}
\label{prop:intro-return-excludes-order}
Let $G/K$ be semiabelian, let $\Ner/\OK$ be its N\'eron lft-model, and let
$P\in G(K)$ have infinite order.  If $n>1$ and
\[
 \mathfrak d_{\Ner}(nP)=\mathfrak d_{\Ner}(P),
\]
then no finite place of semiabelian reduction satisfies $d_v(P)=n$.
\end{proposition}

\begin{proof}
If $d_v(P)=n>1$, then $\overline P_v\ne e_v$, so
$v(\mathfrak d_{\Ner}(P))=0$, whereas $[n]\overline P_v=e_v$, so
$v(\mathfrak d_{\Ner}(nP))>0$.
\end{proof}

Thus a proof of \SilvermanConj{} for a given dense point necessarily
produces infinitely many missing exact orders.  Silverman treated toric and
constant-$j$ function-field cases
\cite{Silverman2004Torus,Silverman2004Elliptic}; further function-field
primitive-divisor results appear in
\cite{IngramMaheSilvermanStangeStreng2012,NaskreckiStreng2020,
GhiocaHsiaTucker2018}.  Barroero--Capuano--Turchet proved the function-field
analogue for abelian and split semiabelian schemes and identified Ribet
sections as the obstruction to their stronger bounded-support conclusion
\cite[Theorem~1.3 and Remark~1.5]{BarroeroCapuanoTurchet2024}.  The
remaining geometrically nonsplit function-field case will be treated
separately in our subsequent work.  

Recently, Bara\'nczuk--Naskr\k{e}cki--Verzobio \cite{BaranczukNaskreckiVerzobio2026} study the squarefree reduction ideals attached to a point \(P\in A(K)\), where \(A/K\) admits a \(K\)-isogeny
\(\phi:A\to E^m\) and the coordinates of
\(\phi(P)=(Q_1,\ldots,Q_m)\) generate a subgroup of \(E(K)\) of
Mordell--Weil rank one. The rank-one coordinate hypothesis forces the cyclic orbit closure of \(P\) to have dimension at most one; hence it is not Zariski dense in \(A\) when \(\dim A\ge2\).  Their points therefore lie outside the hypotheses of Silverman's Conjecture~10.  

Over number fields, Silverman's conjecture
remains open in full generality.

\subsection{Ribet points}
Ribet points originated as deficient points on extensions of abelian varieties
by $\Gm$ \cite{JacquinotRibet1987}.  Their exceptional torsion behavior later
became central in relative Manin--Mumford and unlikely-intersection theory
\cite{Bertrand2011,BertrandMasserPillayZannier2016,
BertrandSchmidt2019,BarroeroKuhneSchmidt2023}.  In the universal setting over
Siegel moduli, the Poincar\'e torsor has a mixed-Shimura interpretation, and
Ribet loci are special subvarieties that are not accounted for by ordinary
subgroup schemes \cite{BertrandEdixhoven2020}.

More precisely, let
\[
 1\longrightarrow\Gm\xrightarrow{\iota}G_q\xrightarrow{\pi}A
 \longrightarrow0
\]
be the extension represented by $q\in A^\vee(K)$.  For a homomorphism
$\beta:A^\vee\to A$, put $\delta:=\beta-\widehat\beta$.  The normalized
Poincar\'e biextension gives a canonical section over the graph of $\delta$;
its value at $q$ is the Ribet point $R_\beta(q)$ above $\delta q$.  For every
$m\ge1$ and every $x\in A^\vee[m]$,
\[
 [m]r_\beta(x)=e_m^A(\beta x,x)\in\mu_m,
 \qquad [m^2]r_\beta(x)=0.
\]
Bertrand--Edixhoven proved the normalized construction over arbitrary bases
and the precise Weil-pairing identity
\cite[Propositions~3.1 and~3.3]{BertrandEdixhoven2020}.

The new arithmetic operation is elementary to state: translate the normalized
Ribet point by a fixed torsion point $t$ in the toric kernel and set
\[
 P=R_\beta(q)+t.
\]
Since
\([d_v(P)]\overline P_v=0\), one has
\[
        [d_v(P)]\overline{R_\beta(q)}_v
        =
        -[d_v(P)]\overline t_v.
\]
The normalized Ribet identity bounds the primary order of
\(\overline{R_\beta(q)}_v\) quadratically in the order of its projection,
whereas the primary order of \(t\) is fixed. 

\subsection*{Main results}

Let $e_\delta$ be the exponent of $\ker\delta$, write
$\ord(t)=\prod_\ell\ell^{a_\ell}$, and define
\[
 \rho_\ell(\delta,t)=
 \max\left\{0,
 \left\lceil\frac{a_\ell-2v_\ell(e_\delta)}2\right\rceil\right\},
 \qquad
 N_{\delta,t}=\prod_\ell\ell^{\rho_\ell(\delta,t)}.
\]
A sharp finite-group amplification lemma forces the explicit integer
$N_{\delta,t}$ to divide $d_v(P)$ at all but finitely many places.  If
$N_{\delta,t}\nmid n$, then $d_v(P)\nmid n$ at every nonexceptional place,
so $nP$ cannot reduce to the identity there.  At the finitely many remaining
places, one chooses $n$ in a suitable congruence class.  Multiplication by
$n$ is then either an automorphism of the formal group at the identity or
keeps the reduction away from the identity, and therefore preserves the
complete local contact multiplicity.  This gives
\[
 \mathfrak d_{\mathcal N}(nP)=\mathfrak d_{\mathcal N}(P).
\]
In this way, the same exceptional torsion-lifting property that makes Ribet
sections special in relative Manin--Mumford theory becomes a mechanism for
exact arithmetic returns of geometric divisibility sequences.
The first main theorem is proved in Theorem~\ref{thm:main}.

\begin{theorem}
\label{thm:intro-main}
Let $A/K$ be a positive-dimensional abelian variety over a number field $K$,
let $q\in A^\vee(K)$, and let
\[
 1\longrightarrow\Gm\xrightarrow{\iota}G_q\xrightarrow{\pi}A
 \longrightarrow0
\]
be the extension represented by $q$.  Let $\beta:A^\vee\to A$ be a
homomorphism, put $\delta=\beta-\widehat\beta$, and assume that $\delta$ is an
isogeny and that $\Z(\delta q)$ is Zariski dense in $A$.

For $t\in\Gm(K)_{\tors}$, identify $t$ with $\iota(t)$ and set
\[
 P=R_\beta(q)+t.
\]
Let $\Ner/\OK$ be the N\'eron lft-model of $G_q$.  Then $G_q$ is
geometrically nonsplit and $\Z P$ is Zariski dense in $G_q$.

If $N_{\delta,t}>1$, there is a finite set $S$ of finite places of $K$ such
that:
\begin{enumerate}[label=\textup{(\roman*)}]
\item
\[
 N_{\delta,t}\mid d_v(P)\qquad(v\notin S);
\]
\item if $N_{\delta,t}\nmid n$, then
\[
 \mathfrak d_{\Ner}(nP)\mathcal O_{K,S}=\mathcal O_{K,S};
\]
\item there is a squarefree integer $Q_P$, divisible by
$\rad(N_{\delta,t})$, such that
\[
 (n,Q_P)=1
 \quad\Longrightarrow\quad
 \mathfrak d_{\Ner}(nP)=\mathfrak d_{\Ner}(P).
\]
\end{enumerate}
\end{theorem}

Here an extension of algebraic groups over $K$ is called
\emph{geometrically nonsplit} if it remains nonsplit after base change to
$\overline K$.

Consequently, all but finitely many rational primes, and every positive power
of each prime outside a fixed finite set, are full return indices and missing
exact reduction orders.  Thus Theorem~\ref{thm:intro-main} is considerably
stronger than the infinitude predicted by \SilvermanConj{}: prime return
indices are cofinite among the rational primes.  The local-to-global argument
is separated from the geometry of a Ribet point in
Proposition~\ref{thm:abstract-shift}, while
Proposition~\ref{prop:finite-cone-return} permits finitely many alternative
divisibility floors.

The second main theorem performs a quadratic descent in which the split
toric kernel is replaced by its norm-one form.  We first fix the two numerical
invariants used in its
statement.  Choose a nonempty open subscheme
$U\subseteq\Spec\mathcal O_K$ over which $G$ extends to a semiabelian scheme
$\mathcal G_U/U$ and $P$ extends to a section.  Define the exact
reduction-order spectrum by
\[
 \mathscr R(G,P):=\{d_v(P):v\in U\}.
\]
Shrinking $U$ removes only finitely many places and hence changes this set by
at most finitely many integers.  For each rational prime $\ell$, put
\[
 u_\ell(P):=
 \sup\left\{
 a\in\mathbf Z_{\ge0}:
 \ell^a\mid d_v(P)
 \text{ for all but finitely many }v\in U
 \right\}
 \in\mathbf Z_{\ge0}\cup\{\infty\}.
\]
If every $u_\ell(P)$ is finite and all but finitely many are zero, define
\[
 U(P):=\prod_\ell\ell^{u_\ell(P)}
\]
and call it the \emph{eventual reduction divisor}.  Equivalently, $U(P)$ is
the greatest positive integer dividing $d_v(P)$ for all but finitely many
$v$.  Both $u_\ell(P)$ and $U(P)$, whenever defined, are unchanged by adding
or deleting finitely many places.

Now, let $L=\Q(\sqrt{-3})$, let
$T=R^1_{L/\Q}\Gm$ be the norm-one torus, and let
\[
 E:\ y^2=x^3-2,
 \qquad Q_0=(3,5),
 \qquad \omega=\frac{-1+\sqrt{-3}}2.
\]
Over $L$, the endomorphism $[\omega]$ has antisymmetric part
$[\sqrt{-3}]$.  Galois conjugation changes the extension class by a sign, and
the sign is exactly compensated by inversion on the split form of $T$. 

\begin{theorem}
\label{thm:intro-surface-Q}
There exist a geometrically nonsplit semiabelian surface
\[
 1\longrightarrow R^1_{L/\Q}\Gm\longrightarrow G\longrightarrow E
 \longrightarrow0
\]
and a point $P_{\mathrm{surf}}\in G(\Q)$ such that
\[
 \overline{\Z P_{\mathrm{surf}}}^{\,\Zar}=G,
 \qquad U(P_{\mathrm{surf}})=2.
\]
There is a squarefree integer $Q_{\mathrm{surf}}$, divisible by $2$, for which
\[
 (n,Q_{\mathrm{surf}})=1
 \quad\Longrightarrow\quad
 \mathfrak d_{\Ner}(nP_{\mathrm{surf}})
 =\mathfrak d_{\Ner}(P_{\mathrm{surf}}).
\]
In particular, all but finitely many rational primes and all their positive
powers are return indices and missing exact reduction orders.
\end{theorem}

This places the phenomenon in dimension two over $\Q$, the smallest dimension
allowed by Silverman's conjecture.  The descent principle is proved in
Theorem~\ref{thm:quadratic-descent}, and the explicit surface in
Theorem~\ref{thm:surface-Q}.

A complementary explicit construction uses the rank-two elliptic curve
\[
 E_{389}:\ y^2+y=x^3+x^2-2x,
 \qquad Q_1=(0,0),\quad Q_2=(1,0).
\]
The resulting geometrically nonsplit threefold over $\Q$ has a dense point
$P$ satisfying
\[
 \mathfrak d_{\Ner}(P)=\Z,
 \qquad d_2(P)=5,
 \qquad U(P)=2.
\]
Hence $D_{nP}=1$ for every $n$ coprime to one fixed integer.  When the
antisymmetric map is an isomorphism, we also obtain the exact local formula
\[
 d_v(P)=m_v\,
 \ord_{k_v^\times}\!\left(
 e_{m_v}^A(\beta\overline q_v,\overline q_v)
 \overline\zeta_v^{\,m_v}
 \right),
 \qquad m_v=\ord(\overline q_v).
\]
This formula determines the eventual reduction divisor in our principal
families.

The applications reveal several phenomena absent from products of abelian
varieties and tori.
\begin{enumerate}[label=\textup{(\alph*)}]
\item For every squarefree $M>1$, after a cyclotomic base extension there is a
      point $P_M$ with Zariski-dense cyclic orbit and $U(P_M)=M$.
\item On the same geometric extension, the unshifted point $R$ has $U(R)=1$,
      while the torsion translate $R+\iota(\zeta_M)$ has $U=M$.
\item The exact-order and primitive-divisor index sets have upper density at
      most $1/M$, while the Zsigmondy set and a bounded-support set have lower
      density at least $1-1/M$.
\item Every positive integer, not necessarily squarefree, can be forced to
      divide almost every reduction order, although exact nonsquarefree
      realization remains open.
\item There are examples over $\Q$ in every dimension at least two, and a
      Zariski-dense set of points of unbounded projected height can share one
      and the same full denominator ideal.
\end{enumerate}
Perucca \cite{Perucca2009} proved that, for a point on a product of an abelian variety and a torus,
the component number of its algebraic orbit closure is the greatest integer
dividing almost every reduction order.  Our orbit closures
are connected, but $U(P_M)=M$ is unbounded.  The missing datum is genuinely
mixed: it is carried by the extension, the biextension lift, and the torsion
translation.

A torus that is nonsplit over $K$ nevertheless becomes split over
$\overline K$; this is distinct from a geometrically nonsplit semiabelian
extension.  For a geometrically dense torus point, Perucca's theorem gives
$U(P)=1$ and excludes every finite covering of almost all reduction orders by
fixed divisibility conditions.  Thus the universal-divisor mechanism of the
present paper is intrinsically mixed.  The dense pure-torus return problem
remains open even in dimension two, to the best of our knowledge; Silverman's
algebraic-integer examples give concrete benchmarks
\cite{Silverman2006AlgebraicIntegers}.

The dependence of $U(P)$ on torsion translation points toward the one-motive
\[
 [\Z\xrightarrow{P}G]
\]
and its affine $\ell$-adic realization.  Jossen's work on Galois and
Mumford--Tate groups of one-motives \cite{Jossen2014} provides a natural framework for such an
intrinsic description.

\medskip
\noindent\textbf{Organization.}
\Cref{sec:denominators} develops full denominator ideals, local freezing,
exact anti-covering, and return criteria, together with the short conditional
Ailon--Rudnick quotient observation.  \Cref{sec:ribet} constructs normalized
Ribet sections and proves the torsion-shift amplification principle and the
main arithmetic theorem.  \Cref{sec:examples} gives the quadratic descent,
the rational surface, the explicit $389\mathrm a1$ threefold, and examples in
every dimension.  The applications to eventual reduction divisors, sparse
exact-order spectra, and primitive divisors are developed alongside the main
theorem.

\subsection*{Acknowledgements}
The author thanks Quang-Khai Nguyen for useful discussions.  He is grateful to
Daniel Bertrand for emphasizing the positive arithmetic role of Ribet points
in the present work and for helpful comments on the abstract; to Joseph H.
Silverman for pointing out Stange's correction of the local elliptic valuation
argument and the exceptional bounded-norm phenomenon; and to Umberto Zannier
for drawing attention to the quantitative refinements and generalizations of
the gcd estimate due to Corvaja--Zannier.  The author gratefully acknowledges
the support of the Morningside Center of Mathematics, Chinese Academy of
Sciences. 

\section{Denominator ideals and local criteria}
\label{sec:denominators}

Throughout, $\N=\{1,2,3,\ldots\}$.  We write the semiabelian group law
additively, but retain multiplicative notation inside a distinguished toric
kernel.  The letter $K$ denotes a number field, $\OK$ its ring of integers,
and $v$ a finite place.  We write
$K_v$, $\cO_v$, $\mathfrak m_v$, and $k_v$ for the completion, valuation
ring, maximal ideal, and residue field.  For nonzero integral ideals
$I,J\subseteq\OK$, our convention is
\[
 I\mid J\quad\Longleftrightarrow\quad J\subseteq I,
 \qquad v(I)=\ord_v(I).
\]

Let $G/K$ be a semiabelian variety and let $\Ner/\OK$ be its N\'eron
lft-model \cite[Chapter~10]{BLR1990}.  Every point of $G(K)$ extends uniquely
to a section of $\Ner$.  Denote the identity section by $e$ and its ideal
sheaf by $\mathcal I_e$.

\begin{definition}[{\cite[Definition~4]{Silverman2005}}]\label{def:denominator}
For $X\in G(K)$ with $X\ne e$ on the generic fiber, define
\[
 \dd_{\Ner}(X):=
 \operatorname{im}\left(
 X^*\mathcal I_e\longrightarrow X^*\mathcal O_{\Ner}=\OK
 \right).
\]
Equivalently, this is the ideal defining the scheme-theoretic inverse image of
the identity section along $X$.  If $K=\Q$, let $D_X$ be its positive
generator:
\[
 D_X\Z=\dd_{\Ner}(X).
\]
\end{definition}

Since the model is separated and the two generic sections are distinct,
$\dd_{\Ner}(X)$ is a nonzero integral ideal.  If
$\dd_{\Ner,v}(X)\subseteq\cO_v$ is its localization, then
\[
 v\bigl(\dd_{\Ner}(X)\bigr)
 =\length_{\cO_v}\bigl(\cO_v/\dd_{\Ner,v}(X)\bigr),
\]
the local scheme-theoretic contact length with the identity.  Whenever
$(\dd_{\Ner}(nP))_{n\ge1}$ is used, $P$ is assumed to have infinite order.
For such a fixed point $P$, define its \emph{full return set} by
\[
 \mathscr T(G,P)
 :=\{n\ge1:\dd_{\Ner}(nP)=\dd_{\Ner}(P)\}.
\]
The elements of $\mathscr T(G,P)$ are called \emph{full return indices}.

\begin{lemma}\label{lem:ideal-divisibility}
If $P\in G(K)$ has infinite order and $m\mid n$, then
\[
 \dd_{\Ner}(mP)\mid\dd_{\Ner}(nP).
\]
\end{lemma}

\begin{proof}
Write $n=mr$.  Since $[r](e)=e$, one has
$[r]^*\mathcal I_e\subseteq\mathcal I_e$.  Pulling back along $mP$ gives
\[
 \dd_{\Ner}(nP)\subseteq\dd_{\Ner}(mP),
\]
which is the asserted ideal divisibility; compare
\cite[Proposition~8]{Silverman2005}.
\end{proof}

Reduction orders are naturally computed on semiabelian enlargement,
whereas the global denominator sequence lives on the N\'eron lft-model.  The following comparison retains complete local multiplicities.

\begin{lemma}
\label{lem:model-comparison}
Let $U\subseteq\Spec\OK$ be a nonempty open subscheme over which $G$ extends
to a semiabelian scheme $\mathcal G_U/U$.  After possibly shrinking $U$, the
canonical morphism
\[
 j:\mathcal G_U\longrightarrow\Ner|_U
\]
is an open immersion onto the maximal open subgroup of $\Ner|_U$ whose fibers
are connected.  In particular, it identifies the formal completions along the
identity sections.  Consequently, for every $X\in\mathcal G_U(U)$ and every
$v\in U$:
\begin{enumerate}[label=\textup{(\roman*)}]
\item the local denominator ideal computed in $\mathcal G_U$ equals the one
      computed in $\Ner$;
\item the order of $\overline X_v$ in $\mathcal G_U(k_v)$ equals its order in
      $\Ner(k_v)$.
\end{enumerate}
\end{lemma}

\begin{proof}
The N\'eron mapping property gives the canonical morphism $j$.  After shrinking
$U$, the standard comparison theorem identifies $\mathcal G_U$ with the open
subgroup $\Ner_U^\circ\subseteq\Ner|_U$ formed by the neutral components of
the fibers; equivalently, it is the maximal open subgroup having connected
fibers over $U$.  See \cite[10.1/7 and 10.1/9]{BLR1990} and the description of
the connected-fiber open subgroup in \cite[\S2, p.~6]{Suzuki2019}.
Since $j$ is an open immersion carrying the identity section of $\mathcal G_U$
to that of $\Ner|_U$, one has scheme-theoretically
\[
 j^{-1}(e_{\Ner})=e_{\mathcal G_U},
 \qquad
 j^*\mathcal I_{e,\Ner}=\mathcal I_{e,\mathcal G_U}.
\]
Pulling this equality back along $X$ proves (i).  Equivalently, if
$\overline X_v\ne e_v$, both local pullback ideals are the unit ideal; if
$\overline X_v=e_v$, the same conclusion, including equality of complete
contact multiplicities, follows from the induced isomorphisms
\[
 \widehat{\mathcal O}_{\mathcal G_U,e_v}
 \simeq\widehat{\mathcal O}_{\Ner,e_v},
 \qquad
 \widehat{\mathcal I}_{e,\mathcal G_U}
 \simeq\widehat{\mathcal I}_{e,\Ner}.
\]
Finally, $\mathcal G_U(k_v)\hookrightarrow\Ner(k_v)$ is an injective group
homomorphism and therefore preserves the exact order of $\overline X_v$,
proving (ii).
\end{proof}

After shrinking $U$, the N\'eron section attached to a fixed $P\in G(K)$
factors through $\mathcal G_U$.  For $v\in U$, put
\[
 d_v(P):=\ord(\overline P_v)
 \quad\text{in }\mathcal G_U(k_v).
\]
The group $\mathcal G_U(k_v)$ is finite, and changing $U$ deletes only
finitely many places.

\begin{lemma}\label{lem:reduction-criterion}
Let $P\in G(K)$ have infinite order, let $v\in U$, and let $n\ge1$.  Then
\[
 v\bigl(\dd_{\Ner}(nP)\bigr)>0
 \quad\Longleftrightarrow\quad
 [n]\overline P_v=e_v
 \quad\Longleftrightarrow\quad
 d_v(P)\mid n.
\]
In particular, if $d_v(P)>1$, then $v(\dd_{\Ner}(P))=0$.
\end{lemma}

\begin{proof}
By Lemma~\ref{lem:model-comparison}, the denominator may be computed in the
finite-type model.  Its pullback ideal lies in $\mathfrak m_v$ exactly when the
reductions of $nP$ and $e$ agree.  The final equivalence is the definition of
the order in the finite group $\mathcal G_U(k_v)$.
\end{proof}

At exceptional places, equality of supports is insufficient: one must preserve the complete contact multiplicity.  The exact local alternatives are as follows.

\begin{lemma}\label{lem:local-freezing}
Let $R$ be a discrete valuation ring with completion $\widehat R$, fraction
field of characteristic zero, and residue characteristic $p>0$.  Let
$\mathcal H/R$ be a smooth separated commutative group scheme locally of
finite type, with identity $e$, and let $X\in\mathcal H(R)$ have infinite
order on the generic fiber.  Put
\[
 o=\ord(\overline X)\in\N\cup\{\infty\}.
\]
Then
\[
 v_R\bigl(([n]X)^*\mathcal I_e\bigr)
 =v_R\bigl(X^*\mathcal I_e\bigr)
\]
in each of the following cases:
\begin{enumerate}[label=\textup{(\roman*)}]
\item $o=\infty$, for every $n\ge1$;
\item $1<o<\infty$ and $o\nmid n$;
\item $o=1$ and $p\nmid n$.
\end{enumerate}
In particular, there is an integer $c_X\ge1$ such that the equality holds
whenever $n\equiv1\pmod{c_X}$; one may take $c_X=o$ in case (ii), $c_X=1$ in
case (i), and $c_X=p$ in case (iii).
\end{lemma}

\begin{proof}
If $o=\infty$, no positive multiple of $\overline X$ is the identity, so all
contact lengths vanish.  If $1<o<\infty$ and $o\nmid n$, then both
$\overline X$ and $[n]\overline X$ are nonidentity, and again both contact
lengths vanish.

Suppose $o=1$.  After base change to $\widehat R$, the section induces a
morphism
\[
 \Spf\widehat R\longrightarrow\widehat{\mathcal H}_e.
\]
If $p\nmid n$, the linear term of multiplication by $n$ is
$n\,\id_{\Lie(\mathcal H/R)}$ and is invertible.  The formal inverse function
theorem makes $[n]$ an automorphism of the formal group, whence
\[
 [n]^*\widehat{\mathcal I}_e=\widehat{\mathcal I}_e.
\]
Pulling back along $X$ gives equality after completion, and faithful flatness
of $R\to\widehat R$ gives equality over $R$.
\end{proof}

\begin{corollary}\label{cor:simultaneous-freezing}
Let $P\in G(K)$ have infinite order and let $S$ be a finite set of finite
places.  There exists $C_S\ge1$ such that
\[
 n\equiv1\pmod{C_S}
 \quad\Longrightarrow\quad
 v\bigl(\dd_{\Ner}(nP)\bigr)=v\bigl(\dd_{\Ner}(P)\bigr)
 \qquad(v\in S).
\]
\end{corollary}

\begin{proof}
Apply Lemma~\ref{lem:local-freezing} at each place and take the least common
multiple of the displayed sufficient moduli.
\end{proof}

The next criterion packages the anti-covering argument in a form that is more flexible than a single divisor.

\begin{proposition}
\label{prop:finite-cone-return}
Let $P\in G(K)$ have infinite order.  Suppose that there are a finite set $S$
of finite places and integers $h_1,\ldots,h_r>1$ such that $G$ has finite-type
semiabelian reduction outside $S$ and
\[
 d_v(P)\in h_1\N\cup\cdots\cup h_r\N
 \qquad(v\notin S).
\]
For $v\in S$, let
\[
 o_v:=\ord(\overline P_v)\in\N\cup\{\infty\},
 \qquad p_v:=\operatorname{char}(k_v),
\]
where the order is taken in the special fiber of the N\'eron lft-model.  Define
\[
 Q:=\rad\!\left(
 h_1\cdots h_r
 \prod_{\substack{v\in S\\1<o_v<\infty}}o_v
 \prod_{\substack{v\in S\\o_v=1}}p_v
 \right).
\]
Then $Q$ is squarefree and
\begin{equation}\label{eq:finite-cone-core}
 (n,Q)=1
 \quad\Longrightarrow\quad
 \dd_{\Ner}(nP)=\dd_{\Ner}(P).
\end{equation}
Consequently $\mathscr T(G,P)$ has lower natural density at least
$\varphi(Q)/Q$; every rational prime $s\nmid Q$ and every $a\ge1$ is a return
index.
\end{proposition}

\begin{proof}
Let $(n,Q)=1$.  If $v\notin S$, then $d_v(P)$ is divisible by some $h_i$,
whereas no prime divisor of $h_i$ divides $n$.  Thus $d_v(P)\nmid n$.
By Lemma~\ref{lem:reduction-criterion}, neither $P$ nor $nP$ meets the identity at
$v$.

Now let $v\in S$.  If $o_v=\infty$, apply
Lemma~\ref{lem:local-freezing}(i).  If $1<o_v<\infty$, then $(n,Q)=1$ implies
$o_v\nmid n$, so apply (ii).  If $o_v=1$, then $p_v\nmid n$, so apply (iii).
The complete local ideals agree at every finite place, proving
\eqref{eq:finite-cone-core}.  The density assertion follows because the
integers coprime to $Q$ have density $\varphi(Q)/Q$; the prime-power assertion
is immediate.
\end{proof}

\begin{remark}
The case $r=1$ is the mechanism used in this paper.  The criterion remains
nontrivial when $\gcd(h_1,\ldots,h_r)=1$.  It therefore suggests a possible
route to pure abelian varieties, where a point with Zariski-dense cyclic orbit on a connected product has no
nontrivial universal divisor in the product setting \cite{Perucca2009}.
\end{remark}

\begin{corollary}[Perucca]
\label{cor:torus-no-finite-multiples}
Let $T/K$ be a positive-dimensional torus over a number field, and let
$P\in T(K)$ have geometrically Zariski-dense cyclic subgroup.  For every
$h_1,\ldots,h_s>1$, there is a set of finite places of positive Dirichlet
density such that
\[
 \gcd\!\left(d_v(P),h_1\cdots h_s\right)=1.
\]
In particular,
\[
 d_v(P)\notin h_1\N\cup\cdots\cup h_s\N
\]
on a set of positive Dirichlet density, and the eventual reduction divisor is
\[
 U(P)=1.
\]
\end{corollary}

\begin{proof}
The algebraic orbit closure of $P$ is the connected torus $T$, so Perucca's
component number is $n_P=1$.  Let $\Sigma$ be the set of rational primes
dividing $h_1\cdots h_s$.  Applying
\cite[Main Theorem~1]{Perucca2009} with auxiliary integer $m=1$ gives a
positive-Dirichlet-density set of places on which
\[
 v_\ell\bigl(d_v(P)\bigr)=0\qquad(\ell\in\Sigma)
\]
simultaneously.  The same theorem identifies $n_P=1$ as the greatest positive
integer dividing almost every reduction order, which proves the assertion
about $U(P)$.
\end{proof}

\begin{remark}[Torsion-valued characters: divisibility versus density]
\label{rem:torus-torsion-character}
A fixed character can force universal divisibility, but only outside the dense
regime.  Suppose that a nonzero
$\chi\in X^*(T_{\overline K})$ satisfies $\chi(P)=\zeta$, where $\zeta$ has
exact order $h>1$.  Choose a finite extension $F/K$ over which $\chi$ and
$\zeta$ are defined.  Outside finitely many places, $\zeta$ retains order $h$
and, for every $w\mid v$ in $F$, one has
$d_w(P)=d_v(P)$ by \cref{lem:base-change-orders}.  Applying $\chi$ to
$[d_w(P)]\overline P_w=e_w$ gives
\[
 \overline\zeta^{\,d_v(P)}=1,
\]
so $h\mid d_v(P)$.  Thus \cref{prop:finite-cone-return} supplies a coprime
core of full returns.

However, $[h]P\in\ker\chi$, and hence
\[
 \Z P\subseteq
 \bigcup_{j=0}^{h-1}\bigl([j]P+\ker\chi\bigr),
\]
a finite union of translates of the proper algebraic subgroup $\ker\chi$.
Therefore $\Z P$ is not geometrically Zariski dense.  The mixed Ribet
mechanism is different: it creates a universal divisor without imposing a
torsion-valued character on the ambient semiabelian variety.
\end{remark}

\subsection{A remark on Ailon--Rudnick's conjecture and Silverman's conjecture}
\label{subsec:ar-conditional-pullback}

We record one functorial consequence of the normalized Ailon--Rudnick
conjecture.  

\begin{lemma}[Functoriality of denominator ideals]
\label{lem:ar-functorial-denominator}
Let $f:G\to H$ be a homomorphism of semiabelian varieties over a number field
$K$, let $\mathcal N_G,\mathcal N_H$ be their N\'eron lft-models, and let
$X\in G(K)$.  Then
\[
        \mathfrak d_{\mathcal N_H}(f(X))
        \subseteq
        \mathfrak d_{\mathcal N_G}(X).
\]
In particular, a unit denominator for $f(X)$ forces a unit denominator for
$X$.
\end{lemma}

\begin{proof}
The N\'eron mapping property extends $f$ to
$\widetilde f:\mathcal N_G\to\mathcal N_H$.  Since
$\widetilde f\circ e_G=e_H$, the scheme-theoretic inverse image of $e_H$
contains $e_G$, and hence
\[
        \operatorname{im}
        (\widetilde f^*\mathcal I_{e_H}\to\mathcal O_{\mathcal N_G})
        \subseteq\mathcal I_{e_G}.
\]
Pullback along $X$ gives the assertion.
\end{proof}

\begin{proposition}
\label{prop:ar-conditional-pullback}
Let $G/K$ be semiabelian, let $P\in G(K)$ have Zariski-dense cyclic orbit,
and suppose there is a homomorphism
\[
        \chi:G\longrightarrow(\Gm^2)_K,
        \qquad \chi(P)=(a,b),
\]
where $a,b\ge2$ are multiplicatively independent integers and
$\gcd(a-1,b-1)=1$.  If
\[
        \gcd(a^n-1,b^n-1)=1
\]
for infinitely many $n$, then
\[
        \mathfrak d_{\mathcal N_G}(P)=\mathcal O_K,
        \qquad
        \mathfrak d_{\mathcal N_G}(nP)=\mathcal O_K
\]
for infinitely many $n$.  
\end{proposition}

\begin{proof}
For $Q=(a,b)$ on $\Gm^2$, a place-by-place calculation on the N\'eron
lft-model gives
\[
        \mathfrak d_{\mathcal N_{\Gm^2}}(nQ)
        =(a^n-1,b^n-1)\mathcal O_K.
\]
At a place where $(v(a),v(b))=(0,0)$ this is the pullback of the identity
ideal on the neutral component; otherwise $nQ$ lies in a component disjoint
from the identity, and at least one displayed generator is a local unit.
The normalization gives a unit denominator at $Q$, and the assumption gives unit denominators at $nQ$.  Apply
\cref{lem:ar-functorial-denominator} to $P$ and to $nP$.
\end{proof}

\begin{remark}
The quotient hypothesis is essential.  If
$1\to\Gm\to G_q\to A\to0$ has non-torsion extension class $q$, then
the exact sequence
\[
        \Hom_K(G_q,\Gm)\longrightarrow\Z
        \xrightarrow{m\mapsto mq}\Ext_K^1(A,\Gm)
\]
and $\Hom_K(A,\Gm)=0$ give $\Hom_K(G_q,\Gm)=0$.  Thus the principal
Ribet extensions below admit no nontrivial toric character quotient.  

For a geometrically dense point on a positive-dimensional pure torus,
Perucca's theorem gives $U(P)=1$ and, more strongly, excludes every finite
covering of almost all reduction orders by fixed sets $h\N$ with $h>1$
\cite{Perucca2009}.  Hence the universal-divisor mechanism of this paper is
genuinely mixed and cannot be transferred directly to dense pure-torus
points.

This boundary is already nontrivial in dimension two.  If $\alpha>1$ is the
real root of $X^3-X-1$ and
\[
 T=\operatorname{Res}^{(1)}_{\Q(\alpha)/\Q}\Gm,
\]
then $T$ is an anisotropic two-dimensional torus and the cyclic subgroup of
$P=\alpha$ is geometrically Zariski dense.  Indeed, the polynomial is
irreducible with discriminant $-23$, so its Galois closure has group $S_3$.
The character lattice
\[
 X^*(T_{\overline\Q})\simeq\Z^3/\Z(1,1,1)
\]
is the irreducible standard representation over $\Q$.  Any nontrivial
torsion-valued relation among the three conjugates of $\alpha$ would therefore
force a positive power of the real conjugate $\alpha>1$ to be a root of
unity, a contradiction.  On the natural integral norm-one model, its
denominator sequence is
\[
 \Delta_n(\alpha)
 :=\max\{d\ge1:\alpha^n\equiv1\pmod{d\mathcal O_{\Q(\alpha)}}\},
\]
which is the sequence denoted $d_n(\alpha)$ by Silverman.  His
Conjecture~9 
\cite[Conjecture~9 and Example~6]{Silverman2006AlgebraicIntegers} predicts $\Delta_n(\alpha)=1$ infinitely often, and this remains
open to the best of our knowledge.  Thus even
the dense nonsplit-torus case is not a formal consequence of
Ailon--Rudnick or of the present method.

Pure abelian varieties lie beyond the quotient criterion for a different
reason: they have no nontrivial algebraic characters.  Accordingly, neither
the normalized Ailon--Rudnick conjecture nor the torsion-shift mechanism
settles Silverman's number-field conjecture for geometrically dense points on
pure abelian varieties of dimension at least two.

A different problem, also not addressed by the present methods, is the
two-parameter elliptic gcd problem over $\Q$: for $\Z$-linearly independent
points $P,Q\in E(\Q)$, the arguments below give no upper bound for
\[
 \gcd\!\bigl(D_{mP},D_{nQ}\bigr)
\]
as $m$ and $n$ vary independently.
\end{remark}

\section{Normalized Ribet sections and applications}\label{sec:ribet}
\subsection{Normalized Ribet sections and quadratic torsion lifting}

Let $S$ be a scheme, let $A/S$ be an abelian scheme, and let $A^\vee$ be its
dual.  Denote by $\cP_A$ the normalized Poincar\'e line bundle on
$A\times_S A^\vee$ and by $\cP_A^\times$ the associated $\Gm$-torsor.
Its biextension structure makes the restriction to $A_T\times\{x\}$, for
$x\in A^\vee(T)$, an extension of $A_T$ by $\Gm$.

Let $\beta:A^\vee\to A$ be a homomorphism, let
$\widehat\beta:A^\vee\to A$ be its dual under biduality, and put
$\delta=\beta-\widehat\beta$.  The antisymmetric difference is exactly
what canonically trivializes the relevant Poincar\'e fiber along its graph.
The resulting normalized section is the geometric source of all subsequent
order bounds.

\begin{proposition}[Normalized Ribet section]
\label{prop:ribet-point}
There is a unique normalized section
\[
 r_\beta\in
 \Gamma\!\left(A^\vee,(\delta,\id)^*\cP_A^\times\right)
\]
whose value at the origin is the unit.  Its formation commutes with arbitrary
base change.  We adopt the Poincar\'e pairing convention for which, for every
$T\to S$, every $m\ge1$, and every $x\in A^\vee[m](T)$,
\begin{equation}\label{eq:ribet-mu}
 [m]r_\beta(x)=e_m^A(\beta x,x)\in\mu_m(T),
\end{equation}
where for $m=1$ the pairing $e_1^A$ is understood to be trivial.  Consequently,
\begin{equation}\label{eq:ribet-square}
 [m^2]r_\beta(x)=0.
\end{equation}
\end{proposition}

\begin{proof}
Poincar\'e functoriality gives canonical biextension isomorphisms
\[
 \cP_A(\beta x,y)
 \simeq\cP_{A^\vee}(x,\widehat\beta y)
 \simeq\cP_A(\widehat\beta y,x).
\]
Setting $y=x$ and using additivity in the first variable canonically
trivializes $\cP_A((\beta-\widehat\beta)x,x)$.  The unit in this
trivialization defines the normalized section; this is
\cite[Proposition~3.1]{BertrandEdixhoven2020}.  If $m=1$, then $x=0$ and
\eqref{eq:ribet-mu} is precisely the normalization $r_\beta(0)=1$, with
$e_1^A$ trivial.  Assume henceforth that $m>1$.  Comparison of the two
canonical $m$-fold biextension trivializations is the Poincar\'e Weil pairing
and gives \eqref{eq:ribet-mu}; see
\cite[Proposition~3.3]{BertrandEdixhoven2020}.  Multiplying once more by $m$
gives \eqref{eq:ribet-square}.
\end{proof}

Over a field $K$, the Barsotti--Weil identification
$\Ext_K^1(A,\Gm)\simeq A^\vee(K)$ is functorial in $A$ and compatible with
pushout in the toric kernel; see \cite[Chapter~VII, \S3]{Serre1988}.
Restriction to $A\times\{q\}$ gives the extension
\begin{equation}\label{eq:extension-Gq}
 1\longrightarrow\Gm\xrightarrow{\iota}G_q\xrightarrow{\pi}A
 \longrightarrow0
\end{equation}
represented by $q\in A^\vee(K)$.  Evaluation of the normalized section gives
\[
 R_\beta(q):=r_\beta(q)\in G_q(K),
 \qquad \pi(R_\beta(q))=\delta q.
\]
This is the Jacquinot--Ribet deficient point associated with $(q,\beta)$;
compare \cite{JacquinotRibet1987,Bertrand2011,
BertrandMasserPillayZannier2016}.

Because the construction is canonical and compatible with base change, the
torsion identity survives specialization without a prime-to-characteristic
restriction.

\begin{corollary}\label{cor:ribet-reduction}
After deleting finitely many places of a number field $K$, all the data above
spread out and the normalized section commutes with reduction.  If
$m_v=\ord(\overline q_v)$, then
\[
 [m_v]\overline{R_\beta(q)}_v
 =e_{m_v}^A(\beta\overline q_v,\overline q_v),
 \qquad
 \ord(\overline{R_\beta(q)}_v)\mid m_v^2.
\]
These are scheme-theoretic identities on the finite flat group scheme
$A^\vee[m_v]$; therefore no prime-to-residue-characteristic hypothesis is
required.
\end{corollary}

\begin{proof}
Spread out the abelian schemes, the Poincar\'e biextension, $\beta$, and the
section $q$.  If $m_v=1$, then $\overline q_v=0$ and normalization gives
$\overline{R_\beta(q)}_v=0$, so both assertions are immediate.  If $m_v>1$,
apply Proposition~\ref{prop:ribet-point} after base change to the residue field.
\end{proof}

\subsection{The torsion-shift amplification principle}\label{sec:finite}

The arithmetic input of the paper is an elementary finite-group argument.  We
state it first in its sharp primewise form and then globalize it to denominator
ideals.  The resulting theorem is independent of Ribet points; the
Poincar\'e-biextension construction will supply its hypotheses in
\cref{sec:main}.

The following primewise inequality is the elementary engine of the paper.
It measures how much of a fixed kernel torsion point must remain visible after
translation by a lift of controlled order.

\begin{lemma}\label{lem:amplification}
Let
\[
 0\longrightarrow T\longrightarrow H\xrightarrow{\pi}B\longrightarrow0
\]
be an exact sequence of finite abelian groups.  Let
\[
 R\in H,\qquad Q=\pi(R),\qquad t\in T,\qquad P=R+t.
\]
Fix a rational prime $\ell$, put $a=v_\ell(\ord(t))$, and suppose that
\begin{equation}\label{eq:order-control}
 v_\ell(\ord(R))
 \le c\,v_\ell(\ord(Q))+b
\end{equation}
for integers $c\ge1$ and $b\ge0$.  Then
\begin{equation}\label{eq:amplification-conclusion}
 v_\ell(\ord(P))
 \ge
 \max\left\{0,\left\lceil\frac{a-b}{c}\right\rceil\right\}.
\end{equation}
\end{lemma}

\begin{proof}
Write
\[
 d=\ord(P),\qquad s=v_\ell(d),\qquad
 u=v_\ell(\ord(Q)),\qquad r=v_\ell(\ord(R)).
\]
Projection of $dP=0$ gives $dQ=0$, hence $u\le s$.  Moreover,
\begin{equation}\label{eq:dR-dt}
 dR=-dt.
\end{equation}
The $\ell$-primary orders of the two sides are
\[
 \ell^{\max(r-s,0)}
 \qquad\text{and}\qquad
 \ell^{\max(a-s,0)},
\]
respectively.  If $s\ge a$, then \eqref{eq:amplification-conclusion} is
immediate.  Suppose $s<a$.  The right-hand side of \eqref{eq:dR-dt} has
positive $\ell$-primary order, so $r>s$ and
\[
 a-s=r-s.
\]
Using \eqref{eq:order-control} and $u\le s$, we obtain
\[
 a-s=r-s\le cu+b-s\le(c-1)s+b.
\]
Thus $a-b\le cs$, which is precisely
\eqref{eq:amplification-conclusion}.
\end{proof}

\begin{remark}
If $\ord(R)$ divides $\ord(Q)^2$, then $c=2$ and $b=0$.  Translation by an
element of order $\ell^{2r-1}$ forces $\ell^r$ to divide $\ord(P)$.  This
scale is sharp under the stated hypotheses.  Indeed, let
\[
 T=\Z/\ell^{2r-1}\Z,\qquad B=\Z/\ell^r\Z,\qquad H=T\oplus B.
\]
If $g$ generates $T$, $u=\ell^{r-1}g$, $Q$ generates $B$, and
\[
 t=(g,0),\qquad R=(-g+u,Q),
\]
then $v_\ell(\ord(R))=2r-1\le2r=2v_\ell(\ord(Q))$, whereas
$P=R+t=(u,Q)$ has order $\ell^r$.  In particular, translation by a nonzero
element of order $\ell$ forces $\ell\mid\ord(P)$.
\end{remark}

We will also use the elementary fact that no cancellation occurs between
commuting elements of coprime order.

\begin{lemma}\label{lem:coprime-orders}
Let $x$ and $y$ be commuting elements of finite orders $a$ and $b$.  If
$\gcd(a,b)=1$, then
\[
 \ord(xy)=ab.
\]
\end{lemma}

\begin{proof}
The order of $xy$ divides $ab$.  If $(xy)^n=1$, then $x^n=y^{-n}$ lies in
$\langle x\rangle\cap\langle y\rangle$, which is trivial because the two
cyclic groups have coprime orders.  Hence $a\mid n$ and $b\mid n$, so
$ab\mid n$.
\end{proof}

We now combine the finite-group estimate with the denominator criterion and
local freezing.  The resulting theorem isolates a reusable arithmetic
principle independently of the Poincar\'e construction.

\begin{proposition}
\label{thm:abstract-shift}
Let
\[
 1\longrightarrow T\longrightarrow G\xrightarrow{\pi}B\longrightarrow0
\]
be an exact sequence of semiabelian varieties over a number field $K$, with
$T$ a torus.  Let $R\in G(K)$, put $Q_0=\pi(R)$, let $t\in T(K)$ be torsion,
and set $P=R+t$.  Assume that $P$ has infinite order.  Suppose that there is a
finite set $S$ such that, for every rational prime $\ell$, one has chosen
integers $c_\ell\ge1$ and $b_\ell\ge0$ satisfying
\begin{equation}\label{eq:abstract-local-control}
 v_\ell\bigl(\ord(\overline R_v)\bigr)
 \le c_\ell v_\ell\bigl(\ord(\overline Q_{0,v})\bigr)+b_\ell
 \qquad(v\notin S).
\end{equation}
Assume also that $t$ retains its generic order in every fiber outside $S$, and
define
\begin{equation}\label{eq:abstract-N}
 N_{\mathbf c,\mathbf b}(t):=\prod_\ell
 \ell^{\max\{0,\lceil(v_\ell(\ord t)-b_\ell)/c_\ell\rceil\}}.
\end{equation}
Only primes dividing $\ord(t)$ contribute.  Then, after enlarging $S$ if
necessary,
\[
 N_{\mathbf c,\mathbf b}(t)\mid d_v(P)
 \qquad(v\notin S).
\]
If $N_{\mathbf c,\mathbf b}(t)>1$, then:
\begin{enumerate}[label=\textup{(\roman*)}]
\item for every $n\ge1$ with $N_{\mathbf c,\mathbf b}(t)\nmid n$,
\[
 \dd_{\Ner}(nP)\cO_{K,S}=\cO_{K,S};
\]
\item there is a squarefree integer $Q$, divisible by
$\rad(N_{\mathbf c,\mathbf b}(t))$, such that
\begin{equation}\label{eq:abstract-coprime-return}
 (n,Q)=1\quad\Longrightarrow\quad
 \dd_{\Ner}(nP)=\dd_{\Ner}(P);
\end{equation}
\item $\mathscr T(G,P)$ has lower natural density at least $\varphi(Q)/Q$, and
for every rational prime $r\nmid Q$ and every $a\ge1$,
\[
 \dd_{\Ner}(r^{a}P)=\dd_{\Ner}(P).
\]
\end{enumerate}
\end{proposition}

\begin{proof}
After enlarging $S$, spread out the exact sequence and the points to compatible
finite-type semiabelian models
\[
 1\longrightarrow\mathcal T_U\longrightarrow\mathcal G_U
 \longrightarrow\mathcal B_U\longrightarrow1
\]
over $U=\Spec\cO_{K,S}$.  For every $v\notin S$, apply
Lemma~\ref{lem:amplification} prime by prime to
\[
 0\longrightarrow\mathcal T_U(k_v)\longrightarrow\mathcal G_U(k_v)
 \longrightarrow
 \im\bigl(\mathcal G_U(k_v)\to\mathcal B_U(k_v)\bigr)
 \longrightarrow0.
\]
This gives $N_{\mathbf c,\mathbf b}(t)\mid d_v(P)$.  If
$N_{\mathbf c,\mathbf b}(t)\nmid n$ and a place $v\notin S$ divided the
denominator of $nP$, then Lemma~\ref{lem:reduction-criterion} would give
$d_v(P)\mid n$, a contradiction.  This proves (i).

For (ii), apply Proposition~\ref{prop:finite-cone-return} with the single cone
$h_1=N_{\mathbf c,\mathbf b}(t)$.  Assertion (iii) is part of that
proposition.
\end{proof}

The density argument required for Ribet points is independent of the
amplification calculation.  A proper connected subgroup surjecting onto the
abelian quotient would split an isogenous pullback of a non-torsion
extension.

\begin{lemma}
\label{lem:dense-lift}
Let
\[
 1\longrightarrow\Gm\longrightarrow G_q\xrightarrow{\pi}A\longrightarrow0
\]
be an extension over a field $K$ of characteristic zero, represented by a
non-torsion point $q\in A^\vee(K)$.  If $P\in G_q(K)$ has Zariski-dense
projection $\pi(P)$ in $A$, then $\mathbf ZP$ is Zariski dense in $G_q$.
\end{lemma}

\begin{proof}
Let
\[
 H=\overline{\mathbf ZP}^{\,\Zar}\subseteq G_q.
\]
Because multiplication and inversion preserve the dense subgroup
$\mathbf ZP$, the closure $H$ is a closed algebraic subgroup.  Since the
projection of $P$ is dense, $\pi(H)=A$.  The quotient
$A/\pi(H^0)$ is connected, as a quotient of the connected abelian variety
$A$, and finite, because it is generated by the image of the finite group
$H/H^0$.  Hence
\begin{equation}\label{eq:dense-projection-H0}
 \pi(H^0)=A.
\end{equation}

Suppose $H^0\ne G_q$.  If $H^0\cap\Gm$ had positive dimension, then it would
contain the whole one-dimensional torus and
\eqref{eq:dense-projection-H0} would force
$\dim H^0=\dim A+1=\dim G_q$, a contradiction.  Thus
\[
 u:=\pi|_{H^0}:H^0\longrightarrow A
\]
is surjective with finite kernel.  It is therefore finite.  Since $A$ is
proper, $H^0$ is proper; hence $H^0$ is an abelian variety and $u$ is an
isogeny.

The inclusion $j:H^0\hookrightarrow G_q$ gives a splitting of the pullback
extension $u^*G_q$ via $x\mapsto(x,j(x))$.  Under
\[
 \Ext_K^1(H^0,\Gm)\simeq(H^0)^\vee(K),
\]
the pullback class is $u^\vee(q)$.  Its vanishing implies $q\in\ker(u^\vee)$.  Since $u^\vee$ is an
isogeny, this kernel is finite, so $q$ is torsion, a contradiction.
Therefore $H^0=G_q$, and hence $H=G_q$.
\end{proof}

\subsection{The arithmetic Ribet-shift theorem}\label{sec:main}

Throughout this subsection, $A/K$ is a positive-dimensional abelian variety.  Let
\[
 \beta:A^\vee\longrightarrow A,
 \qquad
 \delta=\beta-\widehat\beta,
\]
and assume that $\delta$ is an isogeny.  Denote by $e_\delta$ the exponent of
its geometric kernel, namely the least positive integer for which
\begin{equation}\label{eq:kernel-exponent}
 [e_\delta]\ker(\delta)=0.
\end{equation}
Equivalently, multiplication by $e_\delta$ descends through $\delta$: there
is a homomorphism
\[
 \gamma:A\longrightarrow A^\vee
 \quad\text{such that}\quad
 \gamma\circ\delta=[e_\delta]_{A^\vee}.
\]
Since $\delta$ is surjective, one also has
$\delta\circ\gamma=[e_\delta]_A$.

Fix $q\in A^\vee(K)$, put $R=R_\beta(q)$, and let
\[
 t=\iota(\zeta)\in\Gm(K)\subseteq G_q(K)
\]
for a root of unity $\zeta\in K^\times$.  Set $P=R+t$ and write
\[
 h:=\ord(t)=\prod_\ell\ell^{a_\ell}.
\]

The Ribet bound has a canonical primewise form.  The next definition records
exactly the part of each primary component of the translating torsion point
that survives the quadratic loss and the kernel exponent.

\begin{definition}\label{def:Ndelta}
For every rational prime $\ell$, define
\begin{equation}\label{eq:rho-definition}
 \rho_\ell(\delta,t):=
 \max\left\{0,
 \left\lceil\frac{a_\ell-2v_\ell(e_\delta)}{2}\right\rceil
 \right\},
\end{equation}
and set
\begin{equation}\label{eq:N-definition}
 N_{\delta,t}:=\prod_\ell\ell^{\rho_\ell(\delta,t)}.
\end{equation}
Only primes dividing $h$ occur in this finite product.
\end{definition}

\begin{remark}
If $\mathbf Z(\delta q)$ is Zariski dense in $A$, then $\delta(A^\vee)$ is an
abelian subvariety containing a dense subset of $A$, and hence equals $A$.
Since $\dim A^\vee=\dim A$, the homomorphism $\delta$ is automatically an
isogeny.  We retain the hypothesis because the kernel exponent is part of the
quantitative statement.
\end{remark}

We can now state the principal arithmetic theorem.  It combines geometric
nonsplitting and Zariski density with a universal divisor of reduction orders
and a progression of full denominator returns.

\begin{theorem}
\label{thm:main}
Assume that $\mathbf Z(\delta q)$ is Zariski dense in $A$, and let
$\Ner/\OK$ be the N\'eron lft-model of $G_q$.  Then $G_q$ is geometrically
nonsplit, $P$ is non-torsion, and $\mathbf ZP$ is Zariski dense in $G_q$.
These geometric conclusions do not require $N_{\delta,t}>1$.

Assume in addition that $N_{\delta,t}>1$.  Then:
\begin{enumerate}[label=\textup{(\roman*)}]
\item there is a finite set $S$ of finite places such that
\begin{equation}\label{eq:N-divides-orders}
 N_{\delta,t}\mid d_v(P)\qquad(v\notin S);
\end{equation}
\item if $N_{\delta,t}\nmid n$, then
\begin{equation}\label{eq:localized-blanking}
 \dd_{\Ner}(nP)\cO_{K,S}=\cO_{K,S};
\end{equation}
\item there is a squarefree integer $Q=Q(G_q,P,S)$ satisfying
$\rad(N_{\delta,t})\mid Q$ and
\begin{equation}\label{eq:global-coprime-return}
 (n,Q)=1
 \quad\Longrightarrow\quad
 \dd_{\Ner}(nP)=\dd_{\Ner}(P);
\end{equation}
consequently $\mathscr T(G_q,P)$ has lower natural density at least
$\varphi(Q)/Q$ and contains the progression $1+Q\Z_{\ge0}$;
\item for every rational prime $r\nmid Q$ and every $a\ge1$,
\begin{equation}\label{eq:prime-power-return}
 \dd_{\Ner}(r^{a}P)=\dd_{\Ner}(P),
 \qquad
 r^a\notin\mathscr R(G_q,P).
\end{equation}
In particular, all but finitely many rational primes are both full return
indices and missing exact reduction orders.
\end{enumerate}
\end{theorem}

\begin{proof}
The density assumption and $\dim A>0$ imply that $q$ is non-torsion.  Under
$\Ext_K^1(A,\Gm)\simeq A^\vee(K)$, the geometric extension class is
$q_{\overline K}$, which is non-torsion and hence nonzero.  Therefore $G_q$ is
geometrically nonsplit.  Since $\pi(P)=\delta q$ is dense,
Lemma~\ref{lem:dense-lift} proves the geometric assertions.

Enlarge a finite set $S$ and put $U=\Spec\cO_{K,S}$.  Choose abelian schemes
$\mathcal A_U^\vee/U$ and $\mathcal A_U/U$ extending $A^\vee$ and $A$, and a
semiabelian scheme $\mathcal G_U/U$ extending $G_q$, so that $\beta$,
$\widehat\beta$, $\delta$, $q$, the normalized Poincar\'e biextension, the
normalized Ribet section, and $t$ all spread out compatibly over $U$.  Require
also that the extended map
$\delta:\mathcal A_U^\vee\to\mathcal A_U$ be an isogeny whose kernel is killed
by $e_\delta$ and that $t$ retain exact order $h$ in every geometric fiber.

Fix $v\notin S$ and put
\[
 m_v:=\ord(\overline q_v),
 \qquad
 n_v:=\ord(\delta\overline q_v).
\]
Since $[n_v]\overline q_v\in\ker\delta$,
\begin{equation}\label{eq:m-divides-en}
 m_v\mid e_\delta n_v.
\end{equation}
By Corollary~\ref{cor:ribet-reduction},
\begin{equation}\label{eq:R-order-m2}
 \ord(\overline R_v)\mid m_v^2.
\end{equation}
Consequently, for every rational prime $\ell$,
\begin{equation}\label{eq:R-vell-bound}
 v_\ell\bigl(\ord(\overline R_v)\bigr)
 \le2v_\ell(n_v)+2v_\ell(e_\delta).
\end{equation}

Since $H^1(k_v,\Gm)=0$, taking $k_v$-points gives an exact sequence
\[
 0\longrightarrow k_v^\times
 \longrightarrow\mathcal G_U(k_v)
 \longrightarrow\mathcal A_U(k_v)
 \longrightarrow0.
\]
Apply Lemma~\ref{lem:amplification} to this sequence
with
\[
 Q_0=\delta\overline q_v,
 \quad R=\overline R_v,
 \quad t=\overline t_v,
 \quad c=2,
 \quad b=2v_\ell(e_\delta).
\]
It follows that $v_\ell(d_v(P))\ge\rho_\ell(\delta,t)$ for every $\ell$,
proving (i).

The remaining assertions follow from Proposition~\ref{thm:abstract-shift} with these
control data.  More explicitly, (ii) is the reduction-order criterion applied
to \eqref{eq:N-divides-orders}; (iii) is
Proposition~\ref{prop:finite-cone-return} with $h_1=N_{\delta,t}$; and the equality in
(iv) follows from (iii).  The exclusion from $\mathscr R(G_q,P)$ follows from
Proposition~\ref{prop:intro-return-excludes-order}.
\end{proof}

The quantitative invariant has a particularly transparent qualitative
consequence when the kernel exponent and the shift order are coprime.

\begin{corollary}\label{cor:coprime-obstruction}
Let $h=\ord(t)>1$ and assume $\gcd(e_\delta,h)=1$.  Outside a finite set of
places,
\[
 \gcd(n,h)=1\quad\Longrightarrow\quad[n]\overline P_v\ne0.
\]
Equivalently, every reduction order outside that finite set has a nontrivial
common factor with $h$.  If $h$ is prime, then $h\mid d_v(P)$ at every place
outside that finite set.  In particular, the numerical condition $\gcd(\deg\delta,h)=1$ is sufficient.
\end{corollary}

\begin{proof}
Suppose $[n]\overline P_v=0$ with $\gcd(n,h)=1$.  Projection gives
$[n]\overline q_v\in\ker\delta$.  Hence, if
$m_v=\ord(\overline q_v)$, then $m_v\mid e_\delta n$ and therefore
$\gcd(m_v,h)=1$.  The order of $[n]\overline R_v$ divides $m_v^2$ and is
prime to $h$, while $[n]\overline t_v$ has exact order $h$.  They cannot be
negatives of one another.  Taking $n=d_v(P)$ proves the order assertion.  The
exponent $e_\delta$ divides $\deg\delta$, so the last statement follows.
\end{proof}

Over $\Q$, the preceding theorem lies in the literal arithmetic setting of
Silverman's conjecture.

\begin{corollary}
\label{cor:silverman}
If $K=\Q$, then $(\Ner,P)$ satisfies all hypotheses of \SilvermanConj{}.
Moreover, for the squarefree integer $Q$ supplied by Theorem~\ref{thm:main},
\[
 (n,Q)=1\quad\Longrightarrow\quad D_{nP}=D_P.
\]
In particular $D_{r^{a}P}=D_P$ for every rational prime $r\nmid Q$ and every
$a\ge1$.
\end{corollary}

\begin{proof}
The generic fiber is a geometrically connected semiabelian variety of
dimension $\dim A+1\ge2$, hence is irreducible, commutative, and has no
unipotent part.  Density and the return statement are Theorem~\ref{thm:main}; equality
of ideals over $\Z$ is equality of their positive generators.
\end{proof}

Recall that $\mathscr R(G,X)$ denotes the exact reduction-order spectrum of
$X$.  Adding or deleting finitely many places changes this set by at most
finitely many integers.  For $\mathscr S\subseteq\N$, put
\[
 \overline d(\mathscr S)
 :=\limsup_{Y\to\infty}\frac{\#(\mathscr S\cap[1,Y])}{Y},
 \qquad
 \underline d(\mathscr S)
 :=\liminf_{Y\to\infty}\frac{\#(\mathscr S\cap[1,Y])}{Y}.
\]

\begin{proposition}
\label{prop:sparse-orders-general}
Under the hypotheses of Theorem~\ref{thm:main}, assume $N_{\delta,t}>1$.  Then
there is a finite set $F\subset\N$
such that
\[
 \mathscr R(G_q,P)\subseteq N_{\delta,t}\N\cup F.
\]
Consequently,
\[
 \overline d\bigl(\mathscr R(G_q,P)\bigr)
 \le\frac1{N_{\delta,t}},
 \qquad
 \underline d\bigl(\N\setminus\mathscr R(G_q,P)\bigr)
 \ge1-\frac1{N_{\delta,t}}.
\]
For the points $P_M$ constructed in
Theorem~\ref{thm:squarefree-M} below, one has $N_{\delta,t}=M$.
\end{proposition}

\begin{proof}
Outside the finite exceptional set in Theorem~\ref{thm:main}, every reduction order
is divisible by $N_{\delta,t}$; the exceptional places contribute only
finitely many order values.  The density estimates follow because
$N_{\delta,t}\N$ has density $1/N_{\delta,t}$.
\end{proof}

The return core is not only arithmetically large; the corresponding points are
geometrically dense and have unbounded projected height.

\begin{proposition}
\label{thm:gcd-height-decoupling}
Let $P$ satisfy the hypotheses of Theorem~\ref{thm:main}, assume
$N_{\delta,t}>1$, and let $Q$ be a return-core modulus from that theorem.  Then
\[
 \mathcal X_Q:=\{(1+kQ)P:k\ge0\}
\]
is Zariski dense in $G_q$, and
\[
 \dd_{\Ner}(X)=\dd_{\Ner}(P)
 \qquad(X\in\mathcal X_Q).
\]
For every symmetric ample line bundle $\mathcal L$ on $A$,
\[
 \widehat h_{\mathcal L}\bigl(\pi((1+kQ)P)\bigr)
 =(1+kQ)^2\widehat h_{\mathcal L}(\delta q)
 \longrightarrow\infty.
\]
For the explicit rational example, every ideal in the displayed dense family
is the unit ideal.
\end{proposition}

\begin{proof}
The denominator equality follows from $(1+kQ,Q)=1$.  The point $QP$ has dense
projection $Q\delta q$ and hence, by Lemma~\ref{lem:dense-lift}, dense cyclic orbit
in $G_q$.

Let $C$ be the Zariski closure of $\{k(QP):k\ge0\}$.  Translation gives a
descending chain
\[
 C\supseteq QP+C\supseteq2QP+C\supseteq\cdots,
\]
which stabilizes by Noetherianity.  After translating, one obtains
$C=QP+C$, so $-QP\in C$.  Since $C$ is the Zariski closure of an
additive subsemigroup, it is itself closed under addition.  It therefore
contains every positive and negative multiple of $QP$, hence the full cyclic
subgroup $\mathbf Z(QP)$; by density, $C=G_q$.  Translating by $P$ proves the
density of $\mathcal X_Q$.  The height identity is the quadraticity of the
N\'eron--Tate height; it tends to infinity because the dense point $\delta q$
is non-torsion.  The final assertion uses
\eqref{eq:explicit-unit-initial}.
\end{proof}

\subsection{Primitive-divisor and densities}

Let
\[
 \mathcal P_{\mathrm{prim}}(P):=
 \{n\ge1:\dd_{\Ner}(nP)\text{ has a primitive finite-place divisor}\}
\]
and let
\[
 \mathcal Z(P):=\N\setminus\mathcal P_{\mathrm{prim}}(P)
\]
be the Zsigmondy set.  A place is primitive at index $n$ precisely when $n$ is
its first rank of apparition.

\begin{proposition}[Density of primitive-divisor indices]
\label{thm:primitive-density}
Under the hypotheses of Theorem~\ref{thm:main}, assume $N_{\delta,t}>1$.  Then
there is a finite set $F_{\mathrm{prim}}\subset\N$ such that
\[
 \mathcal P_{\mathrm{prim}}(P)
 \subseteq N_{\delta,t}\N\cup F_{\mathrm{prim}}.
\]
Consequently,
\[
 \overline d\bigl(\mathcal P_{\mathrm{prim}}(P)\bigr)
 \le\frac1{N_{\delta,t}},
 \qquad
 \underline d\bigl(\mathcal Z(P)\bigr)
 \ge1-\frac1{N_{\delta,t}}.
\]
For the points $P_M$ of Theorem~\ref{thm:squarefree-M} below, the
right-hand sides are $1/M$ and $1-1/M$, respectively.
\end{proposition}

\begin{proof}
Choose $S$ as in Theorem~\ref{thm:main}.  If $v\notin S$ is primitive at index $n$,
then $d_v(P)=n$, and \eqref{eq:N-divides-orders} gives
$N_{\delta,t}\mid n$.  A fixed place $v\in S$ can be primitive at most one
index: if its reduction has finite order $o_v$, its first appearance is
$o_v$, while a reduction of infinite order never appears.  Thus the places in
$S$ contribute only a finite set of exceptional indices.  The density bounds
follow.
\end{proof}

\begin{proposition}
\label{prop:high-density-support}
Under the hypotheses of Theorem~\ref{thm:main}, assume $N_{\delta,t}>1$.  Then
there is a finite set $S$ of places
such that
\[
 \supp\dd_{\Ner}(nP)\subseteq S
 \qquad\text{whenever }N_{\delta,t}\nmid n.
\]
The set of such indices has natural density
$1-1/N_{\delta,t}$.
\end{proposition}

\begin{proof}
This is exactly \eqref{eq:localized-blanking}; the complement of
$N_{\delta,t}\N$ has the stated density.
\end{proof}

\begin{corollary}
\label{cor:extremal-density}
For every $\varepsilon>0$, there exist a number field $K$, a geometrically
nonsplit semiabelian threefold $G/K$, and a Zariski-dense point $X\in G(K)$
such that
\[
 \overline d\bigl(\mathscr R(G,X)\bigr)<\varepsilon,
 \qquad
 \overline d\bigl(\mathcal P_{\mathrm{prim}}(X)\bigr)<\varepsilon,
\]
\[
 \underline d\bigl(\mathcal Z(X)\bigr)>1-\varepsilon,
\]
and the denominator ideals are supported on one fixed finite set of places for
a set of indices of lower density $>1-\varepsilon$.
\end{corollary}

\begin{proof}
Lemmas~\ref{lem:E389-data} and~\ref{lem:q-dense} below furnish a
symplectic datum over $\Q$ with $\dim B=1$.  Choose a squarefree integer
$M>\max\{1,\varepsilon^{-1}\}$, put $K=\Q(\mu_M)$, and let $X=P_M$ be the
point supplied by Theorem~\ref{thm:squarefree-M} for this datum.  The
corresponding semiabelian extension has dimension
$2\dim B+1=3$.  Its translating torsion point has order $M$; since
$e_\delta=1$ and $M$ is squarefree, one has $N_{\delta,t}=M$.  The asserted
inequalities now follow from
Propositions~\ref{prop:sparse-orders-general},
\ref{thm:primitive-density}, and~\ref{prop:high-density-support}.
\end{proof}

Thus there is no positive lower bound depending only on the dimension for the
density of exact reduction orders or primitive-divisor indices in the class of
geometrically nonsplit semiabelian varieties over varying number fields.

\subsection{Exact reduction orders in the isomorphism case}
\label{sec:exact}

When $\delta:A^\vee\to A$ is an isomorphism, the universal divisibility
theorem admits a complete local refinement.  The quotient order and the
remaining toric factor can then be separated exactly.

\begin{proposition}[Exact Ribet-shift order formula]\label{thm:exact-order}
Assume that $\delta:A^\vee\to A$ is an isomorphism.  Outside a finite set of places, put
\[
 m_v:=\ord(\overline q_v)
     =\ord(\delta\overline q_v)
\]
and define the Weil-pairing factor
\begin{equation}\label{eq:omega-pairing}
 \omega_v:=[m_v]\overline R_v
 =e_{m_v}^A(\beta\overline q_v,\overline q_v)
 \in\mu_{m_v}(k_v).
\end{equation}
Then
\begin{equation}\label{eq:exact-order-formula}
 d_v(P)
 =m_v\,
 \ord_{k_v^\times}\!\left(\omega_v\overline\zeta_v^{\,m_v}\right).
\end{equation}
In particular, $\ord(\omega_v)\mid m_v$.
\end{proposition}

\begin{proof}
The identity \eqref{eq:omega-pairing} is
Corollary~\ref{cor:ribet-reduction}.  The image of $\overline P_v$ in the abelian
quotient has exact order $m_v$.  For an element $x$ of an abelian group whose
image in a quotient has exact order $m$, one has
\[
 \ord(x)=m\,\ord([m]x).
\]
Applying this to $x=\overline P_v$ and using multiplicative notation in the
toric kernel gives
\[
 [m_v]\overline P_v
 =\omega_v\overline\zeta_v^{\,m_v},
\]
which proves \eqref{eq:exact-order-formula}.
\end{proof}

At places where the projected order is coprime to the shift order, the toric
terms have coprime orders and the formula factors numerically.

\begin{corollary}
\label{cor:exact-coprime}
Let $h=\ord(t)$.  At every place covered by Proposition~\ref{thm:exact-order} for which
$\gcd(m_v,h)=1$, one has
\begin{equation}\label{eq:exact-coprime-formula}
 d_v(P)=h\,m_v\,\ord(\omega_v).
\end{equation}
\end{corollary}

\begin{proof}
The element $\omega_v$ has order dividing $m_v$, whereas
$\overline\zeta_v^{m_v}$ has exact order $h$.  Their orders are coprime, so
Lemma~\ref{lem:coprime-orders} and \eqref{eq:exact-order-formula} give the result.
\end{proof}

Formula \eqref{eq:exact-order-formula} separates three pieces of local data:
the order of the projected point, the Weil-pairing contribution of the
normalized Ribet lift, and the fixed toric translation.  A genuine numerical
factorization occurs in Corollary~\ref{cor:exact-coprime}, where the relevant orders are
coprime.

\subsection{A universal symplectic block}\label{sec:block}

The condition that $\beta-\widehat\beta$ be an isomorphism imposes no special
endomorphism hypothesis after doubling a positive-dimensional abelian variety.

\begin{proposition}[Symplectic block]\label{prop:block}
Let $B/K$ be a positive-dimensional abelian variety and put
\[
 A=B\times B^\vee,
 \qquad A^\vee=B^\vee\times B.
\]
Define
\[
 \beta:A^\vee\longrightarrow A,
 \qquad \beta(u,v)=(v,0).
\]
Then
\[
 \widehat\beta(u,v)=(0,u),
 \qquad
 \delta(u,v)=(v,-u).
\]
In particular, $\delta:A^\vee\to A$ is an isomorphism whose inverse
$\delta^{-1}:A\to A^\vee$ is given by
$\delta^{-1}(x,y)=(-y,x)$.
\end{proposition}

\begin{proof}
Under product duality, dualizing the block matrix
$\left(\begin{smallmatrix}0&1\\0&0\end{smallmatrix}\right)$ gives its
transposed dual
$\left(\begin{smallmatrix}0&0\\1&0\end{smallmatrix}\right)$.
The formulas follow immediately.
\end{proof}

Combining the block with a torsion translation gives a flexible family to
which the main theorem applies directly.

\begin{corollary}\label{cor:block-family}
Let $q=(q_1,q_2)\in B^\vee(K)\times B(K)$ have Zariski-dense cyclic orbit in
$A^\vee$, and let $G_q$ and $R_\beta(q)$ be attached to the block map in
Proposition~\ref{prop:block}.  If $\zeta\in K^\times$ is a root of unity of squarefree
order $M>1$, then
\[
 P=R_\beta(q)+\iota(\zeta)
\]
has Zariski-dense cyclic orbit and
\[
 M\mid d_v(P)
\]
outside finitely many places.  Moreover, there is a squarefree integer $Q_P$
with $M\mid Q_P$ such that
\[
 (n,Q_P)=1
 \quad\Longrightarrow\quad
 \dd_{\Ner}(nP)=\dd_{\Ner}(P).
\]
Thus all rational primes outside a fixed finite set, together with all their
positive powers, are full return indices and missing exact reduction orders.
\end{corollary}

\begin{proof}
Here $e_\delta=1$.  For every $\ell\mid M$, one has $a_\ell=1$ and hence
$\rho_\ell(\delta,t)=1$.  Thus $N_{\delta,t}=M$, and Theorem~\ref{thm:main}
applies.
\end{proof}

For the remainder of this subsection, fix the following \emph{symplectic
datum}: a number field $K_0$, a positive-dimensional abelian variety $B/K_0$,
the block map $\beta$ of \cref{prop:block} on
$A=B\times B^\vee$, and a point
\[
        q\in A^\vee(K_0)
\]
whose cyclic subgroup is Zariski dense.  Write $G_q$ for the associated
extension and $R=R_\beta(q)$.  The extension class $q$ is non-torsion, so
$G_q$ is geometrically nonsplit, and $R$ has dense cyclic orbit.  After every
finite base extension we retain the same notation for the base-changed data.

The unshifted and shifted Ribet points lie on this fixed extension and have the
same abelian projection.  Nevertheless their eventual reduction divisors are
radically different.

\begin{theorem}
\label{thm:squarefree-M}
Let $M>1$ be squarefree and put $K_M=K_0(\mu_M)$.  Let
\[
 R=R_\beta(q),
 \qquad
 P_M=R+\iota(\zeta_M),
\]
where $\zeta_M$ is primitive of order $M$.  Then:
\begin{enumerate}[label=\textup{(\roman*)}]
\item $G_q/K_M$ is geometrically nonsplit and both $\mathbf ZR$ and
      $\mathbf ZP_M$ are Zariski dense;
\item the two points have the same projection $\delta q$, but
\begin{equation}\label{eq:torsion-sensitivity}
 U(R)=1,
 \qquad
 U(P_M)=M;
\end{equation}
\item there is a squarefree integer $Q_M$, divisible by $M$, such that
\[
 (n,Q_M)=1
 \quad\Longrightarrow\quad
 \dd_{\Ner}(nP_M)=\dd_{\Ner}(P_M).
\]
\end{enumerate}
\end{theorem}

\begin{proof}
The extension class $q$ is non-torsion and both points project to the dense
point $\delta q$, so Lemma~\ref{lem:dense-lift} proves (i).  Since $\delta$ is an
isomorphism and $M$ is squarefree, Corollary~\ref{cor:block-family} gives universal
divisibility by $M$ and the return-core assertion.

For exactness of $U(P_M)$, apply Perucca's theorem to $q$ and prescribe
valuation zero at every prime dividing $M$.  On a set of places of positive
Dirichlet density, $\gcd(m_v,M)=1$, where
$m_v=\ord(\overline q_v)$.  Since $\ord(\omega_v)\mid m_v$,
Corollary~\ref{cor:exact-coprime} gives
\[
 d_v(P_M)=M\,m_v\ord(\omega_v).
\]
Every prime divisor of $M$ therefore occurs with exact exponent one.  If
$r\nmid M$ is prime, prescribe valuation zero at all primes dividing $Mr$;
on a set of positive Dirichlet density the same formula gives $r\nmid d_v(P_M)$.  Hence
$U(P_M)=M$.

For the unshifted point, Proposition~\ref{thm:exact-order} gives
\[
 d_v(R)=m_v\ord(\omega_v),
 \qquad \ord(\omega_v)\mid m_v.
\]
For every rational prime $\ell$, Perucca's theorem supplies a set of positive Dirichlet density on which
$v_\ell(m_v)=0$, and then $\ell\nmid d_v(R)$.  Thus no rational prime
divides almost every reduction order of $R$, so $U(R)=1$.
\end{proof}

\begin{corollary}
\label{cor:arbitrary-discrepancy}
The ratio
\[
 \frac{U(X)}{\#\pi_0(\overline{\mathbf ZX}^{\,\Zar})}
\]
is unbounded among Zariski-dense points on geometrically nonsplit semiabelian varieties of dimension $2\dim B+1$. 
\end{corollary}

\begin{proof}
In Theorem~\ref{thm:squarefree-M}, the orbit closure is the connected group $G_q$ and
$U(P_M)=M$.  
\end{proof}
Squarefree integers are unbounded, and the construction is the
base change of the same extension over $K_0$. Hence the ambient geometric extension and the abelian projection
may be kept fixed while the numerator ranges over all squarefree integers
after cyclotomic base change. The exact order formula also gives simultaneous local sharpness.

\begin{proposition}
\label{thm:simultaneous-valuations}
Let $M>1$ be squarefree and let $\Sigma$ be any finite set of rational primes
containing every prime divisor of $M$.  There is a set of finite places of
$K_M$ of positive Dirichlet density such that
\[
 v_\ell(d_v(P_M))=
 \begin{cases}
 1,&\ell\mid M,\\
 0,&\ell\in\Sigma,\ \ell\nmid M.
 \end{cases}
\]
\end{proposition}

\begin{proof}
Apply Perucca's theorem to the dense point $q$ and prescribe
$v_\ell(m_v)=0$ for every $\ell\in\Sigma$.  The resulting set of places has
positive Dirichlet density.  There $\gcd(m_v,M)=1$, and
\[
 d_v(P_M)=M\,m_v\ord(\omega_v),
 \qquad \ord(\omega_v)\mid m_v.
\]
The asserted valuations follow.
\end{proof}

\begin{corollary}
\label{cor:pairwise-coprime-cofactors}
There are infinitely many places $v_1,v_2,\ldots$ such that
\[
 d_{v_j}(P_M)=M c_j,
\]
where
\[
 c_j>1,
 \qquad (c_j,M)=1,
 \qquad (c_i,c_j)=1\quad(i\ne j).
\]
\end{corollary}

\begin{proof}
Suppose $v_1,\ldots,v_{j-1}$ have been chosen.  Apply Proposition~\ref{thm:simultaneous-valuations} with $\Sigma$ containing the prime divisors
of $M c_1\cdots c_{j-1}$.  On a set of positive Dirichlet density the normalized cofactor
is coprime to that product.  For every fixed integer $d$, places satisfying
$d_v(P_M)=d$ lie in the finite support of the nonzero ideal
$\dd_{\Ner}(dP_M)$.  Hence, after removing finitely many places, the new
cofactor is $>1$ and differs from the preceding ones.  Induction proves the
claim.
\end{proof}

\begin{proposition}
\label{prop:anti-koblitz}
Under the hypotheses of Theorem~\ref{thm:main}, assume $N_{\delta,t}>1$.  Then only
finitely many finite places have
$d_v(P)$ prime.  More generally, let $R\ge1$ and choose in
Theorem~\ref{thm:squarefree-M} an $M$ having at least $R+1$ distinct prime factors.
Then only finitely many places satisfy
\[
 \Omega(d_v(P_M))\le R,
\]
where $\Omega$ counts prime factors with multiplicity.
\end{proposition}

\begin{proof}
Let $S$ be the finite exceptional set outside which
$N_{\delta,t}\mid d_v(P)$.  If $v\notin S$ and $d_v(P)$ is prime, then
necessarily $d_v(P)=N_{\delta,t}$, and in particular $N_{\delta,t}$ is prime.
Every such $v$ lies in the finite support of the nonzero ideal
$\dd_{\Ner}(N_{\delta,t}P)$.  Adding the finitely many places in $S$ proves the
first assertion.  For the second, $M\mid d_v(P_M)$ outside finitely many places
and $\Omega(M)\ge R+1$.
\end{proof}

\begin{remark}
Koblitz's conjecture concerns primality of the full group order
$\#E(\F_p)$, with refinements accounting for unavoidable fixed divisors and
entanglements \cite{Koblitz1988,Zywina2011}.  The preceding proposition is a
fixed-point anti-Koblitz phenomenon: the unnormalized order $d_v(P)$ is
prime only finitely often.  Whenever $U(P)$ is defined and finite, and after
discarding the finitely many exceptional places at which
$U(P)\nmid d_v(P)$, the natural normalized fixed-point order is
\[
 \frac{d_v(P)}{U(P)}.
\]
For $P_M$, Proposition~\ref{thm:simultaneous-valuations} shows that this normalized
cofactor can avoid every prescribed finite set of primes on a set of positive
Dirichlet density.  This is the expected local admissibility input for a future
Kummer--Chebotarev sieve.
\end{remark}

Squarefreeness is needed for the exact identity $U(P_M)=M$ proved above, but
not for forcing a prescribed integer into almost every reduction order. We show that every integer occurs as a forced divisor.

\begin{theorem}
\label{thm:arbitrary-forced-divisor}
Let $M>1$ and write its prime factorization as
\[
 M=\prod_\ell\ell^{r_\ell},
 \qquad
 H(M):=\prod_{r_\ell>0}\ell^{2r_\ell-1},
\]
where all but finitely many $r_\ell$ are zero.
Over $K=K_0(\mu_{H(M)})$, the base change of $G_q$ is a geometrically
nonsplit semiabelian variety of dimension $2\dim B+1$, and there is a
Zariski-dense point $P^{\mathrm{forc}}_M\in G_q(K)$ such that
\[
 M\mid d_v(P^{\mathrm{forc}}_M)
\]
outside finitely many places.  The eventual divisor exists and satisfies
\begin{equation}\label{eq:forced-U-bounds}
 M\mid U(P^{\mathrm{forc}}_M)\mid H(M).
\end{equation}
Moreover, there is a squarefree return-core modulus $Q_M^{\mathrm{forc}}$ with
\[
 (n,Q_M^{\mathrm{forc}})=1
 \quad\Longrightarrow\quad
 \dd_{\Ner}(nP^{\mathrm{forc}}_M)
 =\dd_{\Ner}(P^{\mathrm{forc}}_M).
\]
\end{theorem}

\begin{proof}
For the fixed symplectic datum, $e_\delta=1$.  Base-change to
$K=K_0(\mu_{H(M)})$ and translate $R$ by
a primitive $H(M)$-th root of unity.  For $\ell\mid M$ the shift exponent is
$a_\ell=2r_\ell-1$, and hence
\[
 \rho_\ell(\delta,t)
 =\left\lceil\frac{2r_\ell-1}{2}\right\rceil=r_\ell.
\]
Thus $N_{\delta,t}=M$, and Theorem~\ref{thm:main} gives universal divisibility and
the return core.

To prove \eqref{eq:forced-U-bounds}, apply Perucca's theorem to obtain a
set of positive Dirichlet density on which $m_v$ is coprime to $H(M)$.
At these places,
Corollary~\ref{cor:exact-coprime} gives
\[
 d_v(P^{\mathrm{forc}}_M)=H(M)m_v\ord(\omega_v),
\]
so $u_\ell(P^{\mathrm{forc}}_M)\le2r_\ell-1$ for $\ell\mid M$.  If a
rational prime $s\nmid H(M)$ is fixed, prescribe valuation zero at the primes
dividing $H(M)s$; the same formula gives
$s\nmid d_v(P^{\mathrm{forc}}_M)$ on a set of positive Dirichlet density.
Hence no other prime contributes to $U$, and the bounds follow.
\end{proof}

\begin{remark}
Determining $U(P^{\mathrm{forc}}_M)$ exactly for nonsquarefree $M$ requires
finer control of cancellation in
$\omega_v\overline\zeta_v^{\,m_v}$.  The theorem proves that every integer is
a universal forced divisor, but exact realization of arbitrary prime powers
remains open.
\end{remark}

\section{Explicit arithmetic examples}\label{sec:examples}

\subsection{Quadratic descent and a semiabelian surface over
\texorpdfstring{$\Q$}{Q}}
\label{sec:descent}

The split-kernel construction above requires the antisymmetric homomorphism
$\delta=\beta-\widehat\beta$ to be defined over the ground field.  For an
elliptic curve over $\Q$, every $\Q$-endomorphism is Rosati symmetric, so a
nonzero split-$\Gm$ block does not exist.  The obstruction disappears after
pairing an anti-invariant CM endomorphism with the sign character of a
nonsplit torus.  The two Galois signs cancel.

Throughout this section, let $L/K$ be a quadratic extension with nontrivial
automorphism $\sigma$, and put
\[
 T=R^1_{L/K}\Gm
 :=\ker\bigl(\operatorname{Nm}_{L/K}:\operatorname{Res}_{L/K}\Gm
 \longrightarrow\Gm\bigr).
\]
After base change to $L$, the torus $T_L$ is isomorphic to $\Gm$, and the
nontrivial Galois element acts on this split form by inversion; see
\cite{Milne2017} for the character-lattice description of tori and Galois
descent.

\begin{lemma}
\label{lem:base-change-orders}
Let $F/K$ be a finite extension, let $G/K$ be a semiabelian variety, and let
$X\in G(K)$.  Outside a finite set of places, for every $w\mid v$ one has
\[
 \ord(\overline X_v)=\ord(\overline X_w).
\]
Consequently $u_\ell(X)$, and hence $U(X)$ whenever defined, is unchanged by
finite base extension.
\end{lemma}

\begin{proof}
After deleting finitely many places, choose a semiabelian model
$\mathcal G$ over an open subscheme of $\Spec\OK$ and base-change it to
$\mathcal O_F$.  For $w\mid v$, the natural map
\[
 \mathcal G(k_v)\longrightarrow\mathcal G(k_w)
\]
is injective.  The order of an element is unchanged under an injective group
homomorphism.  The assertion about the almost-everywhere valuations follows
immediately.
\end{proof}

Let $(A,\lambda)/K$ be a principally polarized abelian variety.  We use
$\lambda$ to identify $A$ with $A^\vee$ and write $\dagger$ for the Rosati
involution.  The Rosati involution commutes with Galois action because
$\lambda$ is defined over $K$.

\begin{lemma}
\label{lem:ribet-inversion}
Let $S$ be a scheme, let $(A,\lambda)/S$ be a principally polarized abelian
scheme, let $\beta\in\End_S(A)$, and put $\delta=\beta-\beta^\dagger$.  Write
$\operatorname{inv}_*$ for pushout along $z\mapsto z^{-1}$ on $\Gm$.  For a
rigidified $\Gm$-torsor $\mathcal T$, write $\mathcal T^\vee$ for its inverse
torsor.  For a section $x\in A(S)$, push the extension $G_{-x}$ out by
inversion on $\Gm$ and use the canonical biextension isomorphism of rigidified
$\Gm$-torsors over $A$,
\begin{equation}\label{eq:poincare-inversion-identification}
 \left(\left.\cP_A^\times\right|_{A\times\{-x\}}\right)^\vee
 \xrightarrow{\sim}
 \left.\cP_A^\times\right|_{A\times\{x\}}.
\end{equation}
The resulting isomorphism
\[
 \jmath_x:\operatorname{inv}_*G_{-x}\xrightarrow{\sim}G_x
\]
satisfies
\begin{equation}\label{eq:ribet-inversion-compatibility}
 \jmath_x\bigl(\operatorname{inv}_*R_{\beta^\dagger}(-x)\bigr)=R_\beta(x),
\end{equation}
where $\operatorname{inv}_*R_{\beta^\dagger}(-x)$ denotes the image of the
point in the pushout.
The construction is canonical and compatible with arbitrary base change.
\end{lemma}

\begin{proof}
The antisymmetric part of $\beta^\dagger$ is $-\delta$, so
\[
 r_{\beta^\dagger}(-x)\in
 \cP_A\bigl((-\delta)(-x),-x\bigr)^\times
 =\cP_A(\delta x,-x)^\times.
\]
Pushout by inversion sends this element to the inverse $\Gm$-torsor of
$\cP_A(\delta x,-x)^\times$.  By biadditivity of the normalized Poincar\'e
biextension in its second variable, \eqref{eq:poincare-inversion-identification}
identifies that inverse torsor canonically with
$\cP_A(\delta x,x)^\times$.  As $x$ varies, the resulting elements form a
section of $(\delta,\id)^*\cP_A^\times$.  Its value at the origin is the unit.
The uniqueness of the normalized Ribet section in
Proposition~\ref{prop:ribet-point} therefore identifies it with $r_\beta$, proving
\eqref{eq:ribet-inversion-compatibility}.
\end{proof}

\begin{theorem}[Quadratic-twisted Ribet descent]
\label{thm:quadratic-descent}
Let $\beta\in\End_L(A_L)$ satisfy
\begin{equation}\label{eq:descent-beta}
 \sigma(\beta)=\beta^\dagger,
\end{equation}
and put
\[
 \delta=\beta-\beta^\dagger.
\]
Assume that $\delta$ is an isogeny.  Let $Q\in A(K)$, put
\[
 q=\delta Q\in A(L),
\]
and assume that the cyclic subgroup generated by $\delta^2Q$ is Zariski dense
in $A$.  Then:
\begin{enumerate}[label=\textup{(\roman*)}]
\item $\sigma(q)=-q$, and the extension $G_q/L$ of $A_L$ by $\Gm$
      represented by $q$ descends to an extension
\begin{equation}\label{eq:descent-extension}
 1\longrightarrow T\longrightarrow G\longrightarrow A\longrightarrow0
\end{equation}
      over $K$;
\item the normalized Ribet point $R_\beta(q)\in G_q(L)$ descends to a point
      $R\in G(K)$ whose projection is $\delta^2Q$;
\item $G$ is geometrically nonsplit, and for every $t\in T(K)_{\tors}$ the
      point $P=R+t$ has Zariski-dense cyclic orbit in $G$;
\item identify $G_L$ with $G_q$ and choose a splitting $T_L\simeq\Gm$.
      Let $e_\delta$ be the exponent of $\ker\delta$ and define
      $N_{\delta,t}$ by \cref{def:Ndelta}.  This integer is independent of the
      splitting: the two splittings differ by inversion and hence replace
      $t$ by $t^{-1}$, which has the same order.  If $N_{\delta,t}>1$, then
\[
 N_{\delta,t}\mid d_v(P)
\]
      outside a finite set of places of $K$.  In particular, there is a
      squarefree integer $Q_P$, divisible by $\rad(N_{\delta,t})$, such that
\begin{equation}\label{eq:descent-return-core}
 (n,Q_P)=1
 \quad\Longrightarrow\quad
 \dd_{\Ner}(nP)=\dd_{\Ner}(P),
\end{equation}
      where $\Ner$ is the N\'eron lft-model of $G$.
\end{enumerate}
\end{theorem}

\begin{proof}
Since Galois action commutes with Rosati,
\[
 \sigma(\beta^\dagger)=\sigma(\beta)^\dagger
 =(\beta^\dagger)^\dagger=\beta.
\]
Hence $\sigma(\delta)=-\delta$, and the $K$-rationality of $Q$ gives
$\sigma(q)=-q$.

Regard $q$ as an element of $A_L^\vee(L)$ through the fixed principal
polarization.  Under the functorial Barsotti--Weil identification
\cite[Chapter~VII, \S3]{Serre1988}, the conjugate extension $\sigma^*G_q$ has
class $\sigma(q)=-q$.  Let
\[
 c:\sigma^*G_q=G_{-q}\xrightarrow{\sim}G_q
\]
be the canonical isomorphism obtained by first pushing $G_{-q}$ out along
$z\mapsto z^{-1}$ and then applying the rigidified Poincar\'e identification
\eqref{eq:poincare-inversion-identification}; equivalently, it is the composite
\[
 G_{-q}\longrightarrow\operatorname{inv}_*G_{-q}
 \xrightarrow{\ \jmath_q\ }G_q.
\]
It induces the identity on $A_L$ and inversion on the toric kernel.

We verify the cocycle rather than appealing implicitly to uniqueness.  The
composite
\[
 c\circ\sigma^*c:\sigma^{2*}G_q\longrightarrow G_q
\]
induces the identity both on the abelian quotient and on the toric kernel,
because inversion is applied twice.  Automorphisms of an extension inducing
the identity on kernel and quotient are parametrized by
$\Hom_L(A_L,\Gm)$, which is zero.  Therefore
\[
 c\circ\sigma^*c=\id_{G_q}.
\]
Thus $c$ is a descent datum.

An algebraic group of finite type over a field is quasi-projective
\cite[Lemma~39.8.7, Tag~0BF7]{StacksQuasiProjective}.  Effective descent for
quasi-projective schemes along the surjective finite locally free map
$\Spec L\to\Spec K$ is
\cite[Lemma~39.25.3, Tag~0CCJ]{StacksDescent}.  It descends the underlying
scheme.  The multiplication, inverse, identity, quotient map, and kernel
embedding are compatible with the descent datum and descend by fpqc descent
for morphisms.  This gives the group extension
\eqref{eq:descent-extension}.  Its toric kernel is the norm-one torus because
the descent action on the split character lattice is the sign action.

We next descend the normalized Ribet point.  Base-change functoriality of the
Poincar\'e biextension and the normalized section gives
\[
 \sigma\bigl(R_\beta(q)\bigr)=R_{\beta^\dagger}(-q)
 \quad\text{in }G_{-q}(L).
\]
The isomorphism $c$ is precisely pushout by inversion on the toric fiber,
followed by the canonical Poincar\'e biextension identification.  Therefore
Lemma~\ref{lem:ribet-inversion}, applied with $x=q$, gives
\begin{equation}\label{eq:descent-ribet-identity}
 c\!\left(\sigma(R_\beta(q))\right)=R_\beta(q).
\end{equation}
Thus $R_\beta(q)$ is fixed by the descent datum and descends to $R\in G(K)$.
Its projection is
\[
 \delta q=\delta^2Q.
\]
Notice that $\delta^2$ is Galois invariant and hence defined over $K$.

The density hypothesis implies that $\delta^2Q$, and therefore $q$, is
non-torsion.  After base change to $L$, the extension class of $G$ is the
non-torsion point $q$.  Hence the extension is geometrically nonsplit.  The
projection of $R+t$ is $\delta^2Q$, so Lemma~\ref{lem:dense-lift} applied over
$L$ shows that $R+t$ is Zariski dense in $G_L$, and therefore in $G$.

Finally apply Theorem~\ref{thm:main} after base change to $L$.  It gives
$N_{\delta,t}\mid d_w(P)$ outside finitely many places $w$ of $L$.
By Lemma~\ref{lem:base-change-orders}, the same divisibility holds outside
finitely many places $v$ of $K$.  The return-core assertion now follows from
Proposition~\ref{prop:finite-cone-return} over $K$.
\end{proof}

The descent theorem has a particularly simple realization on a $j=0$ elliptic
curve.  It gives the phenomenon in the smallest dimension allowed by
Silverman's conjecture and over the original ground field $\Q$.

\begin{theorem}
\label{thm:surface-Q}
Let
\[
 L=\Q(\sqrt{-3}),
 \qquad
 T=R^1_{L/\Q}\Gm,
 \qquad
 E:\ y^2=x^3-2,
\]
and put
\[
 \omega=\frac{-1+\sqrt{-3}}2,
 \qquad
 \beta=[\omega]\in\End_L(E_L),
 \qquad
 Q_0=(3,5)\in E(\Q).
\]
Then there is a geometrically nonsplit semiabelian surface
\begin{equation}\label{eq:surface-extension}
 1\longrightarrow T\longrightarrow G_{\mathrm{surf}}
 \longrightarrow E\longrightarrow0
\end{equation}
and a point $P_{\mathrm{surf}}\in G_{\mathrm{surf}}(\Q)$ such that:
\begin{enumerate}[label=\textup{(\roman*)}]
\item $\mathbf ZP_{\mathrm{surf}}$ is Zariski dense in
      $G_{\mathrm{surf}}$;
\item
\[
 U(P_{\mathrm{surf}})=2;
\]
\item there is a squarefree integer $Q_{\mathrm{surf}}$, divisible by $2$,
      such that
\begin{equation}\label{eq:surface-return-core}
 (n,Q_{\mathrm{surf}})=1
 \quad\Longrightarrow\quad
 \dd_{\Ner}(nP_{\mathrm{surf}})
 =\dd_{\Ner}(P_{\mathrm{surf}});
\end{equation}
\item all but finitely many rational primes, together with all their positive
      powers, are return indices and missing exact reduction orders.
\end{enumerate}
Consequently Silverman's conjecture holds for the pair
$(G_{\mathrm{surf}},P_{\mathrm{surf}})$, and the product-case equality between
the component number of the orbit closure and the eventual reduction divisor
does not extend to semiabelian surfaces over $\Q$.
\end{theorem}

\begin{proof}
The curve $E$ has complex multiplication by $\Z[\omega]$ over $L$: explicitly,
$[\omega](x,y)=(\omega x,y)$.  Under the canonical principal polarization,
the Rosati involution acts as complex conjugation on the CM field; see, for
example,
\cite[Proposition~5.1]{Deligne1982}.  Thus
\[
 \beta^\dagger=[\overline\omega]=[\omega^2],
 \qquad
 \delta:=\beta-\beta^\dagger=[\omega-\omega^2]=[\sqrt{-3}],
\]
so
\begin{equation}\label{eq:surface-delta}
 \delta^2=[-3],
 \qquad \deg\delta=3.
\end{equation}
The geometric kernel has order $3$ and is nontrivial; hence its exponent is
\begin{equation}\label{eq:surface-kernel-exponent}
 e_\delta=3.
\end{equation}
Complex conjugation satisfies $\sigma(\beta)=\beta^\dagger$.

The point $Q_0$ is non-torsion.  Indeed, it is integral and nonzero, and the
Nagell--Lutz theorem would force either $y(Q_0)=0$ or
$y(Q_0)^2\mid\Delta(E)=-1728$; but $25\nmid1728$
\cite[Chapter~VIII, Corollary~7.2]{SilvermanAEC}.  Set
\[
 q=\delta Q_0.
\]
Then $\sigma(q)=-q$ and
\[
 \delta q=\delta^2Q_0=[-3]Q_0
\]
is non-torsion, hence has dense cyclic orbit on $E$.
Theorem~\ref{thm:quadratic-descent} therefore descends both the extension and the
normalized Ribet point to $\Q$.

The norm-one torus contains the rational point $t=-1$ of exact order $2$.
Let $R\in G_{\mathrm{surf}}(\Q)$ be the descended Ribet point and set
\[
 P_{\mathrm{surf}}=R+t.
\]
Since $v_2(e_\delta)=v_2(3)=0$, the defining formula gives
\[
 \rho_2(\delta,t)=\left\lceil\frac12\right\rceil=1,
 \qquad N_{\delta,t}=2.
\]
Thus Theorem~\ref{thm:quadratic-descent} gives geometric nonsplitting, density,
universal evenness outside finitely many places, and
\eqref{eq:surface-return-core}.  The final return and missing-order statement
follows from Proposition~\ref{prop:intro-return-excludes-order}.

It remains to prove exactness of the eventual divisor.  We work first over
$L$.  At a good finite place $w$, put
\[
 m_w=\ord(\overline q_w).
\]
If $(m_w,3)=1$, then multiplication by $\delta$ preserves the order of
$\overline q_w$: its image order divides $m_w$, while
$m_w\mid3\ord(\delta\overline q_w)$ by \eqref{eq:m-divides-en}.  Hence the
projection of $\overline P_{\mathrm{surf},w}$ has exact order $m_w$.
If in addition $m_w$ is odd, the Ribet identity gives
\[
 [m_w]\overline R_w=\eta_w\in\mu_{m_w}(k_w),
 \qquad \ord(\eta_w)\mid m_w,
\]
and $[m_w]t=t$.  The two factors $\eta_w$ and $t$ have coprime orders, so
\[
 \ord(\eta_w\,t)=2\ord(\eta_w).
\]
The image of $\overline P_{\mathrm{surf},w}$ in the elliptic quotient has
exact order $m_w$.  For any element $x$ whose image in a quotient has exact
order $m$, one has $\ord(x)=m\ord([m]x)$.  Therefore
\begin{equation}\label{eq:surface-exact-order}
 \begin{aligned}
 d_w(P_{\mathrm{surf}})
 &=m_w\ord\!\left([m_w]\overline P_{\mathrm{surf},w}\right)\\
 &=2m_w\ord(\eta_w)
 \end{aligned}
 \qquad\bigl((m_w,6)=1\bigr).
\end{equation}
All places in this calculation are understood to lie outside the fixed finite
set containing those above $2$ and $3$ and those where the curve, extension,
descent datum, or torsion section has bad reduction.

The orbit closure of $q$ is the connected elliptic curve $E$.  In Perucca's
notation its component number is therefore $n_q=1$, and we apply
\cite[Main Theorem~1]{Perucca2009} with $m=1$.  It gives sets of positive
Dirichlet density on which the valuations of $m_w$ at any prescribed finite
set of rational primes are all zero.  Choosing $m_w$ coprime to $6$
in \eqref{eq:surface-exact-order} shows that
$v_2(d_w(P_{\mathrm{surf}}))=1$ on a set of positive Dirichlet density.  Given any odd
prime $r$, choosing $m_w$ coprime to $6r$ shows that
$r\nmid d_w(P_{\mathrm{surf}})$ on a set of positive Dirichlet density.  Hence
\[
 U(P_{\mathrm{surf},L})=2.
\]
By Lemma~\ref{lem:base-change-orders}, $U$ is invariant under the finite extension
$L/\Q$, proving (ii).  Since the orbit closure is the connected group
$G_{\mathrm{surf}}$, its component number is $1$, which proves the last
assertion.
\end{proof}

A $\Q$-defined endomorphism of an elliptic curve is fixed by the Rosati
involution, so the split-$\Gm$ construction has no nonzero antisymmetric block
over $\Q$; see Proposition~\ref{prop:no-surface-Q-method} below.  The theorem above does not
contradict that obstruction.  The endomorphism $[\omega]$ and the extension
class are defined only over $L$ and are both anti-invariant, while the
character lattice of the norm-one torus carries the same sign.  Descent pairs
the two signs.  Thus the true obstruction is split-toric, not dimensional.

\begin{proposition}
\label{prop:no-surface-Q-method}
Let $E/\Q$ be an elliptic curve, identify $E\simeq E^\vee$ through the
canonical principal polarization, and let $\beta:E^\vee\to E$ be defined over
$\Q$.  Then $\beta=\widehat\beta$.  Consequently, the split-$\Gm$
Ribet-shift construction of \cref{thm:main} cannot produce a semiabelian
surface over $\Q$ with elliptic quotient.
\end{proposition}

\begin{proof}
Via the principal polarization, $\beta$ corresponds to an element of
$\End_\Q(E)$.  In the non-CM case $\End_\Q(E)=\Z$.  In the CM case, the
differential representation
\[
 \End_\Q(E)\longrightarrow
 \End_\Q H^0(E,\Omega^1_{E/\Q})\simeq\Q
\]
is injective in characteristic zero.  Hence a $\Q$-defined CM endomorphism is
an element of the imaginary quadratic CM field lying in $\Q$; being integral
over $\Z$, it lies in $\Z$.  Thus every $\Q$-endomorphism is a multiplication
map.  Multiplication maps are fixed by the Rosati involution, so
$\beta=\widehat\beta$.
\end{proof}

\begin{remark}
For comparison, the direct split-$\Gm$ construction also gives a surface after
a CM base extension.  On
\[
 E_0:\ y^2=x^3-2x,
 \qquad q_0=(2,2),
 \qquad K=\Q(\zeta_8),
\]
take $\beta=[i]$ and translate by $\zeta_8$.  Then
$\delta=[2i]$, $e_\delta=2$, and $N_{\delta,t}=2$.  The point $q_0$ is
non-torsion.  Indeed, if $q_0$ were torsion, then
$[2]q_0$ would also be torsion, and the Nagell--Lutz theorem would force
$[2]q_0$ to have integral coordinates.  This contradicts
\[
 [2]q_0=\left(\frac94,-\frac{21}{8}\right);
\]
see \cite[Chapter~VIII, Corollary~7.2]{SilvermanAEC}.  The new
Theorem~\ref{thm:surface-Q} is stronger in that both the surface and the point are
defined over $\Q$.
\end{remark}

\subsection{A geometrically nonsplit threefold over
\texorpdfstring{$\Q$}{Q}}

Consider the elliptic curve
\begin{equation}\label{eq:E389}
 E/\Q:\qquad y^2+y=x^3+x^2-2x
\end{equation}
and the rational points
\[
 Q_1=(0,0),\qquad Q_2=(1,0).
\]

We begin the explicit construction by recording the arithmetic information
needed to prove independence and density.

\begin{lemma}[Arithmetic of $389\mathrm a1$]\label{lem:E389-data}
One has
\[
 E(\Q)=\Z Q_1\oplus\Z Q_2,
 \qquad E(\Q)_{\tors}=0,
 \qquad \End_{\overline\Q}(E)=\Z.
\]
Moreover,
\[
 N_E=389,\qquad \Delta(E)=389,\qquad c_4(E)=112,
 \qquad j(E)=\frac{112^3}{389}.
\]
\end{lemma}

\begin{proof}
The rank, trivial torsion subgroup, displayed Mordell--Weil basis, and
conductor are the certified data for the curve $389\mathrm a1$ in Cremona's
tables \cite{Cremona1997,CremonaData}.  The displayed invariants follow from
the standard formulas for the integral Weierstrass equation
\eqref{eq:E389}.  Since $j(E)\in\Q\setminus\Z$, it is not an algebraic
integer.  Every CM $j$-invariant is an algebraic integer
\cite[Theorem~11.1]{Cox2013}; therefore $E$ has no complex multiplication and
$\End_{\overline\Q}(E)=\Z$.
\end{proof}

Put $A=E\times E$ and identify $A\simeq A^\vee$ by the product principal
polarization.  Let
\[
 q=(Q_1,Q_2)\in A^\vee(\Q),
\]
and define
\[
 \beta(X,Y)=(Y,0),
 \qquad
 \widehat\beta(X,Y)=(0,X).
\]
Thus
\begin{equation}\label{eq:explicit-delta}
 \delta(X,Y)=(Y,-X),
 \qquad \delta^2=[-1].
\end{equation}
Let
\[
 1\longrightarrow\Gm\xrightarrow{\iota}G_q
 \xrightarrow{\pi}E^2\longrightarrow0
\]
be the extension represented by $q$, let $R=R_\beta(q)$, and put
\begin{equation}\label{eq:explicit-P}
 \eps=\iota(-1),
 \qquad P=R+\eps.
\end{equation}

The Mordell--Weil basis and the absence of extra geometric endomorphisms now
force the extension parameter, and hence its antisymmetric image, to be
Zariski dense.

\begin{lemma}\label{lem:q-dense}
The cyclic subgroups generated by $q$ and by
$\delta q=(Q_2,-Q_1)$ are Zariski dense in $E^2$.
\end{lemma}

\begin{proof}
It is enough to treat $\delta q$, since $\delta$ is an isomorphism.  Let
\[
 H=\overline{\mathbf Z(\delta q)}^{\,\Zar},
 \qquad B=H^0,
 \qquad c=\#(H/B).
\]
If $B=0$, then $\delta q$ is torsion, contrary to
Lemma~\ref{lem:E389-data}.  Suppose that $B$ is a proper positive-dimensional
abelian subvariety of $E^2_{\overline\Q}$.  Then $B$ is an elliptic curve,
and the quotient $E^2/B$ is an elliptic curve isogenous to $E$.  Composing the
quotient map with an isogeny $E^2/B\to E$ gives a nonzero homomorphism
$E^2\to E$.  Since $\End_{\overline\Q}(E)=\Z$, it has the form
\[
 f_{a,b}:E^2\longrightarrow E,
 \qquad f_{a,b}(X,Y)=[a]X+[b]Y,
\]
for some $(a,b)\in\Z^2\setminus\{(0,0)\}$, and
$B\subseteq\ker f_{a,b}$.  Since
$[c]\delta q\in B$,
\[
 0=f_{a,b}([c]\delta q)=[ca]Q_2-[cb]Q_1.
\]
The group $E(\Q)$ is torsion-free and $Q_1,Q_2$ form a basis, so $a=b=0$,
a contradiction.  Hence $B=E^2$ and $H=E^2$.
\end{proof}

The data
\[
 K_0=\Q,\qquad B=E,\qquad A=E\times E,
 \qquad q=(Q_1,Q_2)
\]
therefore instantiate the symplectic datum fixed in
\cref{sec:block}: by \cref{lem:q-dense}, the cyclic subgroup generated by
$q$ is Zariski dense in $A^\vee\simeq E^2$.  Thus every structural family
constructed from that datum above admits this completely explicit
specialization over $\Q$.

We can therefore assemble the preceding ingredients into an explicitly
specified Zariski-dense instance of Silverman's return phenomenon.

\begin{theorem}
\label{thm:explicit-Q}
Let $\Ner/\Z$ be the N\'eron lft-model of $G_q$, and put
\[
 U_0=\Spec\Z\!\left[\frac1{389}\right],
 \qquad
 U=\Spec\Z\!\left[\frac1{2\cdot389}\right].
\]
Then:
\begin{enumerate}[label=\textup{(\roman*)}]
\item $G_q$ is a geometrically nonsplit, geometrically integral semiabelian
      threefold, and $P\in G_q(\Q)=\Ner(\Z)$ has Zariski-dense cyclic orbit;
\item the construction extends to a semiabelian scheme
      $\mathcal G_{U_0}/U_0$, and
\begin{equation}\label{eq:explicit-order-two}
 d_2(P)=5;
\end{equation}
      over $U$, for every prime $p\notin\{2,389\}$ and every odd integer
      $n\ge1$,
\begin{equation}\label{eq:explicit-odd-obstruction}
 [n]\overline P_p\ne0,
\end{equation}
      and consequently
\begin{equation}\label{eq:explicit-even}
 2\mid d_p(P)\qquad(p\notin\{2,389\});
\end{equation}
\item the initial full denominator ideal is the unit ideal:
\begin{equation}\label{eq:explicit-unit-initial}
 \dd_{\Ner}(P)=\Z,
 \qquad D_P=1;
\end{equation}
      for every odd $n$, the ideal $\dd_{\Ner}(nP)$ is supported on
      $\{2,389\}$, and
\[
 v_2\bigl(\dd_{\Ner}(nP)\bigr)>0
 \quad\Longleftrightarrow\quad 5\mid n;
\]
\item there is a squarefree integer $Q_P$, divisible by $10$, such that
\begin{equation}\label{eq:explicit-coprime-return}
 (n,Q_P)=1\quad\Longrightarrow\quad D_{nP}=1;
\end{equation}
      hence $D_{r^{a}P}=1$ for every rational prime $r\nmid Q_P$ and every
      $a\ge1$, and every such $r^a$ is missing from the exact-order spectrum.
\end{enumerate}
Thus $(\Ner,P)$ satisfies all hypotheses and a cofinite prime-index strengthening of
Silverman's Conjecture~10.
\end{theorem}

\begin{proof}
The extension class is the non-torsion point $q$, so $G_q$ is geometrically
nonsplit.  A semiabelian variety is smooth and geometrically connected, hence
geometrically integral.  By Lemma~\ref{lem:q-dense} and Equation~(\ref{eq:explicit-delta}), the
projection $\delta q$ is dense and $\delta$ is an isomorphism; therefore
Theorem~\ref{thm:main} gives the density of $P$.

The curve $E$ has discriminant $389$, so $E$, $A=E^2$, the points
$Q_1,Q_2$, the product polarization, $\beta$, $\widehat\beta$, and $\delta$
extend over $\Z[1/389]$.  The normalized Poincar\'e torsor, the extension
represented by $q$, and the normalized Ribet section extend over the same
base.  This gives $\mathcal G_{U_0}$.  After inverting $2$, the toric section
$-1$ has exact order $2$ in every geometric fiber.

We first compute the reduction at $2$.  Direct enumeration gives
\[
 E(\F_2)=\{O,(0,0),(0,1),(1,0),(1,1)\},
 \qquad \#E(\F_2)=5.
\]
Both $\overline Q_1$ and $\overline Q_2$ are nonidentity, hence have exact
order $5$.  Therefore $\overline q_2$ and
$\delta\overline q_2$ have exact order $5$.  The Ribet identity gives
\[
 [5]\overline R_2
 =e_5^A(\beta\overline q_2,\overline q_2)
 \in\mu_5(\F_2)=\{1\}.
\]
The shift $-1$ reduces to $1$ in $\F_2^\times$, so
$[5]\overline P_2=0$.  Since the projection of $\overline P_2$ has exact
order $5$, this proves \eqref{eq:explicit-order-two}.

Now fix $p\notin\{2,389\}$ and put
$m_p=\ord(\overline q_p)$.  The reduction of $\delta$ is an isomorphism, so
$\delta\overline q_p$ also has exact order $m_p$.  If an odd $n$ killed
$\overline P_p$, then projection would give $m_p\mid n$, so $m_p$ would be
odd.  The reduced Ribet point has order dividing $m_p^2$, and hence
$[n]\overline R_p$ has odd order.  On the other hand,
\[
 [n]\overline R_p=-[n]\overline\eps_p=-\overline\eps_p
\]
has exact order $2$, a contradiction.  This proves
\eqref{eq:explicit-odd-obstruction}; taking $n=d_p(P)$ proves
\eqref{eq:explicit-even}.

We next prove \eqref{eq:explicit-unit-initial}.  The projection of $P$ is
\[
 \pi(P)=\delta q=(Q_2,-Q_1).
\]
At every prime $p\ne389$, both coordinates have integral affine reduction and
therefore their pair is not the identity of $E^2(\F_p)$.  At $p=389$, the
displayed Weierstrass equation is minimal because $v_{389}(\Delta)=1$.
Writing
\[
 F(x,y)=y^2+y-x^3-x^2+2x,
\]
we have
\[
 \frac{\partial F}{\partial y}(Q_2)=1,
 \qquad
 \frac{\partial F}{\partial y}(-Q_1)=-1
 \pmod{389}.
\]
Thus $Q_2=(1,0)$ and $-Q_1=(0,-1)$ specialize to smooth affine points of the
minimal special fiber.  The smooth locus of a minimal Weierstrass model maps
into the N\'eron model, with the point at infinity as identity
\cite[Chapter~VII, \S2]{SilvermanAEC}; hence the pair
$(\overline Q_2,-\overline Q_1)$ is nonidentity in the special fiber of the
N\'eron model of $E^2$.

The quotient homomorphism $G_q\to E^2$ extends, by the N\'eron mapping
property, to the N\'eron lft-models.  If $P$ met the identity section at any
prime, its projection would do the same.  The preceding paragraph rules this
out at every prime, so $\dd_{\Ner}(P)=\Z$.  The support assertion for odd $n$
now follows from the parity argument away from $\{2,389\}$, and the criterion
at $2$ follows from $d_2(P)=5$.

Finally apply Proposition~\ref{prop:finite-cone-return} with $S=\{2,389\}$ and
$h_1=2$.  At $2$ the local order is $5$, while at $389$ the reduction is
nonidentity, so its order is either an integer $>1$ or infinite.  The resulting
squarefree return-core modulus is divisible by $2\cdot5=10$ and gives
\eqref{eq:explicit-coprime-return}.  The final exact-order assertion follows
from Proposition~\ref{prop:intro-return-excludes-order}.
\end{proof}

\begin{remark}
Direct calculations verify that $Q_1,Q_2$ lie on the curve
$389\mathrm a1$, compute
\[
 \Delta(E)=389,
 \qquad c_4(E)=112,
 \qquad c_6(E)=-856,
 \qquad j(E)=\frac{112^3}{389},
\]
check that $\delta^2=-I_2$, enumerate $E(\mathbf F_2)$, and verify that
$\overline q_2$ and $\delta\overline q_2$ have exact order $5$.  They also
show that the reductions of $Q_2$ and $-Q_1$ at $389$ are smooth affine
nonidentity points.  For the rational surface, direct calculations verify that
$(3,5)$ lies on $y^2=x^3-2$, compute the discriminant $-1728$, and check the
Nagell--Lutz divisibility obstruction.  Together with the Ribet identity and
$\mu_5(\mathbf F_2)=\{1\}$, these calculations verify the finite arithmetic
inputs used in the two explicit constructions.  The Mordell--Weil basis is
taken from Cremona's certified tables and is not inferred from a bounded
computation.
\end{remark}

\begin{remark}
Away from $2$ and $389$, the obstruction is purely parity-theoretic: an odd
multiple of the Ribet term has odd order and cannot cancel the nontrivial
$2$-torsion shift.  At $2$ the shift specializes to the identity, and the
reduction order is instead the odd value $5$.
\end{remark}

We now give the primewise formulation of the eventual reduction divisor. Perucca proved two results that must be distinguished.  For a general
semiabelian variety, the component number of the algebraic orbit closure
divides almost every reduction order \cite[Proposition~2.2]{Perucca2009}.  For
a product of an abelian variety and a torus, this component number is the
greatest such universal divisor, and finitely many $\ell$-adic valuations may
be prescribed on a set of places of positive Dirichlet density
\cite[Main Theorem~1]{Perucca2009}.  In every application below, the
extension parameter $q$ has Zariski-dense cyclic orbit in a connected abelian
variety.  Thus the component number denoted $n_q$ by Perucca is $1$, and we
apply her Main Theorem with the auxiliary integer $m=1$.  Our examples show
that equality with the component number fails maximally for geometrically
nonsplit extensions.

\begin{theorem}
\label{thm:U-equals-2}
For the point $P$ of Theorem~\ref{thm:explicit-Q},
\[
 U(P)=2.
\]
More precisely,
\[
 u_2(P)=1,
 \qquad u_\ell(P)=0\quad(\ell\text{ odd}),
\]
and the set of primes $p$ for which $v_2(d_p(P))=1$ has positive Dirichlet
density.
\end{theorem}

\begin{proof}
Equation \eqref{eq:explicit-even} gives $u_2(P)\ge1$.  Let
$m_p=\ord(\overline q_p)$ and let $\omega_p$ be the Weil-pairing factor from
Proposition~\ref{thm:exact-order}.  The orbit closure of $q$ in $E^2$ is connected.
Perucca's theorem, with prescribed $2$-adic valuation zero, gives a set of
primes of positive Dirichlet density for which $m_p$ is odd.  At such a prime,
Corollary~\ref{cor:exact-coprime} with $h=2$ gives
\[
 d_p(P)=2m_p\ord(\omega_p),
\]
and both factors after $2$ are odd.  Thus $v_2(d_p(P))=1$.

Fix an odd prime $\ell$.  Prescribe valuation zero at $2$ and $\ell$ in
Perucca's theorem.  On a set of positive Dirichlet density,
$\gcd(m_p,2\ell)=1$.  Since $\ord(\omega_p)\mid m_p$, the same formula gives
$\ell\nmid d_p(P)$.  Hence $u_\ell(P)=0$.
\end{proof}

\begin{corollary}
\label{cor:component-failure}
For the rational surface of Theorem~\ref{thm:surface-Q},
\[
 \overline{\mathbf ZP_{\mathrm{surf}}}^{\,\Zar}=G_{\mathrm{surf}},
 \qquad
 \#\pi_0\!\left(
 \overline{\mathbf ZP_{\mathrm{surf}}}^{\,\Zar}
 \right)=1,
 \qquad
 U(P_{\mathrm{surf}})=2.
\]
Thus Perucca's general component-divisibility lower bound can be strict even
for a point with Zariski-dense cyclic orbit on a geometrically nonsplit
semiabelian surface over $\Q$.  The explicit split-kernel threefold of Theorem~\ref{thm:explicit-Q}
provides a second rational example with the same strict inequality.
\end{corollary}

\subsection{Examples over \texorpdfstring{$\Q$}{Q} in every dimension at least two}

The rational surface of Theorem~\ref{thm:surface-Q} is the minimal-dimensional
example.  Adjoining split toric factors produces examples in every larger
dimension without destroying either density or the return core.  The only
possible obstruction would be a common toric quotient; it is excluded by the
absence of rational characters on the surface.

\begin{lemma}
\label{lem:no-characters}
For the surface $G_{\mathrm{surf}}/\Q$ of Theorem~\ref{thm:surface-Q},
\[
 \Hom_\Q(G_{\mathrm{surf}},\Gm)=0.
\]
\end{lemma}

\begin{proof}
Let $\chi:G_{\mathrm{surf}}\to\Gm$ be a character.  Its restriction to the
norm-one torus $T=R^1_{L/\Q}\Gm$ is a $\Q$-character of $T$.  The character
lattice $X^*(T)$ is a rank-one lattice on which
$\Gal(L/\Q)$ acts by $-1$, so $X^*(T)^{\Gal(L/\Q)}=0$ and
$\Hom_\Q(T,\Gm)=0$.  Hence $\chi$ factors through the elliptic quotient
$E$, but $\Hom_\Q(E,\Gm)=0$.  Thus $\chi=0$.
\end{proof}

\begin{lemma}[Commutative Goursat lemma]\label{lem:goursat}
Let $G$ be a connected commutative algebraic group over a field of
characteristic zero and let $T_0$ be a split torus.  Assume
$\Hom(G,\Gm)=0$.  If a connected algebraic subgroup
$H\subseteq G\times T_0$ projects surjectively onto both factors, then
$H=G\times T_0$.
\end{lemma}

\begin{proof}
Put
\[
 N_G=\{g\in G:(g,1)\in H\},
 \qquad
 N_{T_0}=\{t\in T_0:(1,t)\in H\}.
\]
The commutative Goursat construction gives
$G/N_G\simeq T_0/N_{T_0}$.  If $H$ were proper, this common quotient would be
a nontrivial connected quotient of a split torus, hence a nontrivial split
torus.  Composing the quotient map from $G$ with a nonzero character would
contradict $\Hom(G,\Gm)=0$.
\end{proof}

We conclude by the observation on cofinite prime-index returns in every dimension at least two.
\begin{theorem}
\label{thm:all-dimensions}
For every $d\ge2$, there exist a geometrically nonsplit connected
semiabelian variety $G_d/\Q$ of dimension $d$, a point $P_d\in G_d(\Q)$ with
Zariski-dense cyclic orbit, and a squarefree integer $Q_d$ such that
\[
 (n,Q_d)=1\quad\Longrightarrow\quad D_{nP_d}=D_{P_d}.
\]
Consequently
\[
 D_{r^{a}P_d}=D_{P_d}
\]
for every $a\ge1$ and every rational prime $r\nmid Q_d$; all these prime
powers are missing exact reduction orders.
\end{theorem}

\begin{proof}
For $d=2$, take $(G_{\mathrm{surf}},P_{\mathrm{surf}})$ from
Theorem~\ref{thm:surface-Q}.  Let $d>2$, choose multiplicatively independent
$a_1,\ldots,a_{d-2}\in\Q^\times$, and put
\[
 G_d=G_{\mathrm{surf}}\times\Gm^{d-2},
 \qquad
 P_d=(P_{\mathrm{surf}},a_1,\ldots,a_{d-2}).
\]
Let $H$ be the Zariski closure of $\mathbf ZP_d$.  Its projections onto the
surface and the split torus are surjective.  Since $H/H^0$ is finite, the
quotients of the two connected target groups by the corresponding projections
of $H^0$ are both connected and finite, hence trivial.  Thus $H^0$ projects
surjectively onto both factors.  By
Lemmas~\ref{lem:no-characters} and~\ref{lem:goursat}, $H^0=G_d$, so $P_d$ is dense.

Outside a finite set of places, the order of the surface component is even by
Theorem~\ref{thm:surface-Q}; therefore $2\mid d_v(P_d)$ outside a finite set.
Apply Proposition~\ref{prop:finite-cone-return} with $h_1=2$.  Viewed as an extension
of $E$ by $T\times\Gm^{d-2}$, the class of $G_d$ has nonzero projection to the
$T$-coordinate, namely the class of $G_{\mathrm{surf}}$.  A geometric splitting
of $G_d$ would therefore split $G_{\mathrm{surf}}$, a contradiction.  Thus
$G_d$ is geometrically nonsplit.  The missing-order statement follows from
Proposition~\ref{prop:intro-return-excludes-order}.
\end{proof}

\bibliographystyle{amsalpha}
\bibliography{Revise}

\end{document}